\documentclass[12pt,english]{article}
\usepackage[T1]{fontenc}
\usepackage[latin9]{inputenc}
\usepackage{geometry}
\usepackage{float}
\usepackage{textcomp}
\usepackage{url}
\usepackage{amsmath}
\usepackage{amsthm}
\usepackage{amssymb}
\usepackage{stmaryrd}
\usepackage{graphicx}

\makeatletter

\providecommand{\tabularnewline}{\\}

\theoremstyle{definition}
\newtheorem*{defn*}{\protect\definitionname}
\theoremstyle{plain}
\newtheorem{thm}{\protect\theoremname}
\theoremstyle{plain}
\newtheorem{prop}[thm]{\protect\propositionname}
\theoremstyle{plain}
\newtheorem*{cor*}{\protect\corollaryname}
\theoremstyle{definition}
\newtheorem*{xca*}{\protect\exercisename}
\theoremstyle{remark}
\newtheorem*{rem*}{\protect\remarkname}

\usepackage{amsmath}
\usepackage{url}
\usepackage{listings}
\usepackage{url}
\makeatother

\usepackage{babel}
\providecommand{\corollaryname}{Corollary}
\providecommand{\definitionname}{Definition}
\providecommand{\exercisename}{Exercise}
\providecommand{\propositionname}{Proposition}
\providecommand{\remarkname}{Remark}
\providecommand{\theoremname}{Theorem}

\begin{document}
\title{Inverse spherical Bessel functions generalize Lambert $W$ and solve
similar equations containing trigonometric or hyperbolic subexpressions
or their inverses}
\author{David R. Stoutemyer}
\maketitle
\begin{abstract}
A strict integer Laurent polynomial in a variable $x$ is 0 or a sum
of one or more terms having integer coefficients times $x$ raised
to a negative integer exponent. Equations that can be transformed
to certain such polynomials times $\exp(-x)=\mathit{constant}$ are
exactly solvable by inverses of modified spherical Bessel functions
of the second kind $k_{n}(x)$ where $n$ is the order, generalizing
the Lambert $W$ function when $n>0.$ Equations that can be converted
to certain such polynomials times $\cos(x)$ or such polynomials times
$\sin(x)$ or a sum thereof $=\mathit{constant}$ are exactly solvable
by inverses of spherical Bessel functions $y_{n}(x)$ or $j_{n}(x)$.
Such equations include $\cos(x)/x=\mathit{constant}$, for which the
solution $\mathrm{inverse}\,_{1}(y_{0})(-constant)$ is the Dottie
number when $\mathit{constant}=1$, where subscript $1$ is the branch
number. Equations that can be converted to certain strict integer
Laurent polynomials times $\sinh(x)$ and possibly also plus such
a polynomial times $\cosh(x)$ are exactly solvable by inverses of
modified spherical Bessel functions of the first kind $i_{n}(x)$.

These discoveries arose from the AskConstants program surprisingly
proposing the explicit exact closed form solution $\mathrm{inverse}\,_{1}(y_{0})(-1)$
for the approximate input 0.739085133215160642, because no explicit
exact closed form representation was known for this Dottie number
from approximately 1865 to 2022. This article includes descriptions
of how to implement these spherical Bessel functions and their multi-branched
real inverses.
\end{abstract}

\section{Introduction}

Professor Dottie noticed that whenever she typed any real number $x_{0}$
into her calculator in radian mode then repeatedly pressed the cosine
button, the result always converged to the value
\[
\cos\cos\ldots\cos x_{0}\rightarrow0.73908513321516\,.
\]
This rounded value of the Dottie number (OEIS A003957, \cite{OEIS})
is the only real fixed point of 
\begin{equation}
x=\cos x.\label{eq:DottiesFixedPointEqn}
\end{equation}
This transcendental number was known at least as long ago as 1865
\cite{Bertrand}. Hansha \cite{OzanerHansha} nicely summarizes some
facts about it.

Figure \ref{Figure:IntersectionOfxWithCosx} plots the intersecting
two sides of Dottie's fixed-point equation (\ref{eq:DottiesFixedPointEqn}). 

\begin{figure}[H]
\caption{The one and only real intersection of $x$ with $\cos x$}
\label{Figure:IntersectionOfxWithCosx}
\noindent \centering{}\includegraphics{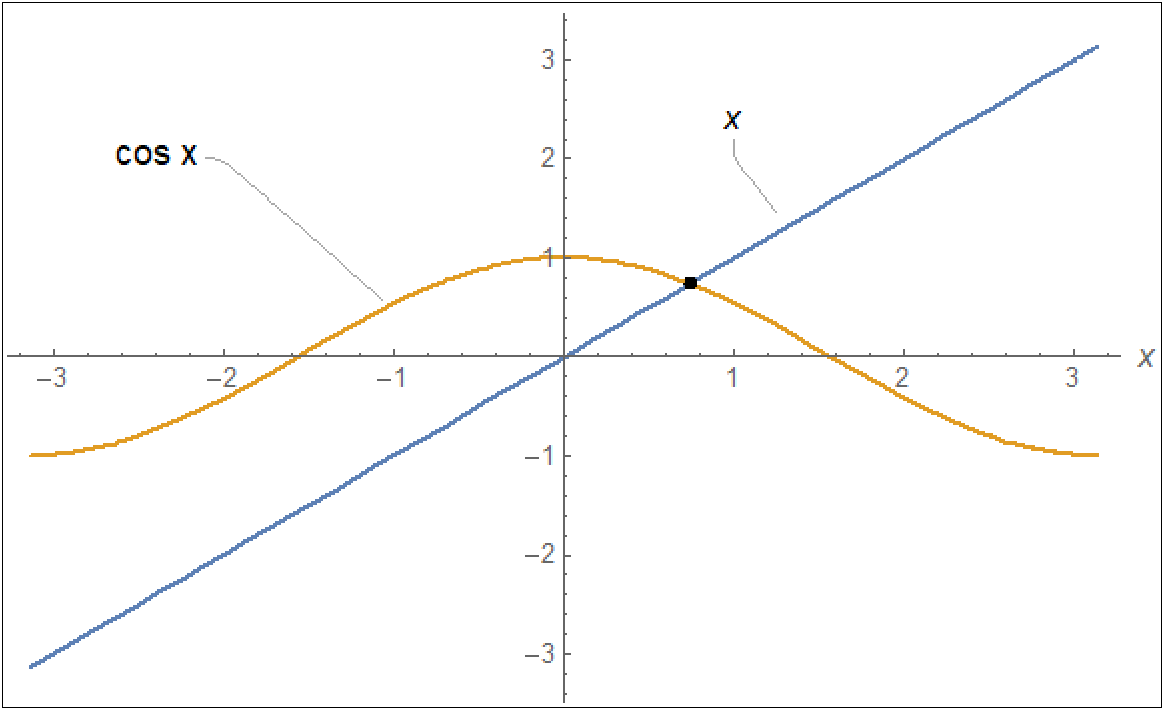}
\end{figure}

Mathematica can express this number as $\mathsf{Root}\,[\{\#1\,-\,\mathsf{Cos}\,[\#1]\,\&,0.73908513321516\}]\}$.
This representation can be considered exact because Mathematica can
do some exact computations with it, such as
\[
\begin{array}{c}
\mathrm{\mathsf{With}}\,[\left\{ x=\mathrm{\mathsf{Root}}\,\left[\left\{ \#1-\mathrm{\mathsf{Cos}}[\#1]\,\&,\,0.73908513321516\right\} \right]\right\} ,\\
\mathrm{\mathsf{FullSimplify}}\,\left[x-\mathrm{\mathsf{Cos}}\,[x]\right],
\end{array}
\]
which returns exact \textsl{integer} 0 rather than merely a floating-point
residual that has small magnitude compared to $0.73908513321516\,.$
For brevity, \textbf{floating-point} \textbf{numbers} are hereinafter
called \textbf{floats}.

Moreover, $\mathrm{\boldsymbol{\mathrm{\mathsf{N}}}}\left[\mathrm{\mathsf{Root\,}}[{\#1-\mathrm{\mathsf{Cos}}\,[\#1]\,\,\&,\,0.73908513321516}],\:\boldsymbol{precision}\right]$
returns a float having the requested value of $\mathit{precision}$
despite the float in the Root expression having only 14 significant
digits. For example on my 4 gigahertz computer,
\[
\mathrm{\boldsymbol{\mathrm{\mathsf{N}}}}\left[\mathrm{\mathsf{Root\,}}[{\#1-\mathrm{\mathsf{Cos}}\,[\#1]\,\,\&,\,0.73908513321516}],\:\boldsymbol{1000}\right]
\]
returns 1000 digits in 0.002 seconds. In contrast, $\mathrm{\boldsymbol{\mathrm{\mathsf{N}}}}\left[0.73908513321516,1000\right]$
returns 0.73908513321516 unchanged because there is insufficient information
in 0.73908513321516 alone to extend it. The finite-precision float
in the $\mathsf{Root}$ expression is merely to indicate the closest
root with sufficient precision to guarantee that the same root can
be extended to any finite precision up to the limits of your patience
and computer memory.

Although such Root expressions can be considered exact, they are implicit
rather than explicit representations, making them aesthetically less
satisfying than solutions of the form $x=\mathit{constant}$ where
$\mathit{constant}$ is composed of a finite number of rational numbers,
symbolic constants such as $\pi$, operators, and functions having
known names, including subexpressions of the form $\mathrm{inverse}\,(f)(\mathit{constant})$,
where $f$ is a known name, such as erf.\footnote{Just as nouns enjoy the economy of concisely representing descriptive
phrases, explicitly named inverses such as $\arctan$ or Mathematica's
InverseErf are mentally more economical than subexpressions containing
compositions such as $\mathrm{inverse}\,(f)$, which are more economical
than natural language descriptions such as ``the least positive solution
to the equation \ldots ''.
\noindent \begin{flushright}
``\textsl{What's in a name}?''\\
\textendash{} Shakespeare
\par\end{flushright}} Functional forms such as $\intop_{\ldots}^{\cdots}\ldots dx$ and
$\sum_{\ldots}^{\cdots}\ldots$ do not occur, having been replaced
by explicit closed-form equivalents.

As a test, I copied 18 digits of the approximate value of the Dottie
number into Version 5.0 of my AskConstants application \cite{StoutemyerAskConsants}
and \cite{StoutemyerHowToHuntWildConstants}. I merely hoped that
it did not propose an impostor exact closed form and misleadingly
assess it a high likelihood of being a correct limit as the precision
approaches $\infty$. However, I was surprised and delighted to discover
that the application returned a simple candidate that is easily proved
to be a previously unknown explicit exact closed form solution. Figure
\ref{Figure:AskConstantsDottie} shows that AskConstants proposes
\[
\mathrm{\mathsf{RealInverseSphericalBesselY}}\,\,\left[0,-1,1\right]
\]
as the limit.\footnote{$\mathrm{\mathsf{RealInverseSphericalBesselY}}\,\,\left[n,t,b\right]$
is one of about 45 AskConstants functions that implement real inverses
of Mathematica functions that have no builtin inverse counterparts.
Here $n=0$ is the order of the inverted function $y_{n}$, $t=-1$
is the value of $y_{n}$, and $b=1$ is the branch number of the countably
infinite branch numbers of the inverse. The $\mathsf{Real}$ prefix
is because the implementation is designed only for real values of
$t$ where the inverse is also real, which is all that is necessary
for AskConstants (and for many other applications).} Such constant-``recognition'' programs generate conjectures rather
than proofs. Section \ref{sec:SolvableViayn} contains a proof, and
the ordinate of the large dot in the scatter plot encouragingly indicates
that this candidate agrees with the input to all 18 of the entered
significant digits.

\begin{figure}[H]
\noindent \begin{centering}
\caption{AskConstants proposes an explicit exact closed form for the Dottie
number}
\label{Figure:AskConstantsDottie}
\par\end{centering}
\includegraphics[scale=0.5]{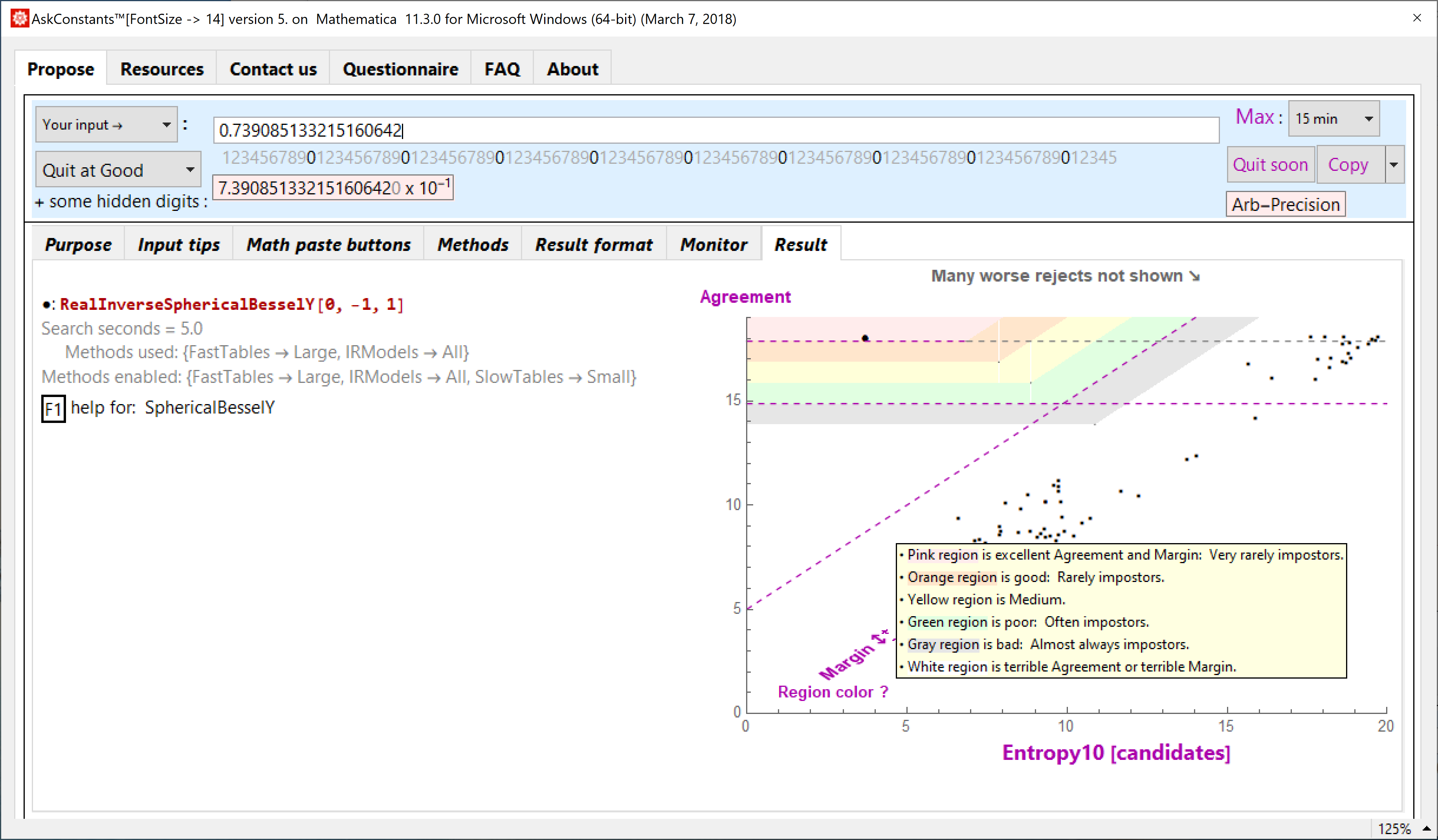}
\end{figure}

However, it is easy to match any number of float digits with a sufficiently
complicated non-float expression. For example, all floats are \textsl{exactly}
representable as rational numbers \textendash{} albeit often with
many digits in the numerator and/or denominator. Therefore AskConstants
also assesses Agreement discounted by a complexity measure that also
has digits as units: The scatter plot also shows that the abscissa
Entropy10 complexity measure of the proposed candidate is approximately
4.0. The Entropy10 of an exact expression is the sum of the base 10
logarithms of the absolute values of the nonzero integers in numerators
and denominators in the expression, plus an average of approximately
1.0 per operator, function or symbolic constant. The Likelihood of
a proposed exact candidate expression being the true limit increases
with
\[
\mathrm{Margin}:=\mathrm{Agreement}-\mathrm{Entropy10\,.}
\]
The Margin is thus approximately 18.0 - 4.0 = 14.0 digits, which together
with the Agreement loss of approximately 0.0 digits makes the candidate
very likely to be a true limit.\footnote{And entering all 108 digits from OEIS A003957, \cite{OEIS} gives
the same explicit closed-form result agreeing to 108 digits, which
is much stronger evidence.} 

The smaller scatter plot dots in Figure \ref{Figure:AskConstantsDottie}
are the best of many rejects, and mousing over them invokes tooltips
with their candidate formulas. For example, the leftmost small dot
above the upper dashed horizontal line is for the candidate
\[
\dfrac{2(55103+19462\mathrm{\mathsf{\,Khinchin}})}{290541}
\]
with $\mathsf{Khinchin}$ representing his constant $\approxeq2.68545$.
Its Agreement is also approximately all 18 digits, but with Entropy10
approximately 17.0, giving it an agreement Margin of only approximately
1.0. As indicated in the plot legend, a plot point must be in a color
band near the pink end of the spectrum in the upper left corner of
the plot to have a high likelihood of being the limit of the float
as its Precision approaches infinity.

Bill Gosper expressed similar surprise when he subsequently obtained
the same candidate exact value for the Dottie number. This article
is an explanation together with some additional discoveries made while
formulating the explanation. More specifically:

Section \ref{sec:SolvableViayn} shows how the inverses of spherical
Bessel functions $y_{n}(x)$ can solve some equations transformable
to the form
\[
\left(\sum_{\ell=1}^{n+1}\dfrac{c_{\ell}}{x^{\ell}}\right)\cos x+\left(\sum_{\ell=1}^{n}{\displaystyle \dfrac{s_{\ell}}{x^{\ell}}}\right)\sin x=c_{0}
\]
for finite $n\geq0$ where $c_{\ell}$ and $s_{\ell}$ are particular
integers and $c_{0}$ is any constant of any kind in the range of
the left side. Dottie's equation (\ref{eq:DottiesFixedPointEqn})
can be transformed to this form for $n=0$$\,$.

Section \ref{sec:SolvableViajn} shows how the inverses of spherical
Bessel functions $j_{n}(x)$ can similarly solve some equations transformable
to the form
\[
\left(\sum_{\ell=1}^{n+1}\dfrac{s_{\ell}}{x^{\ell}}\right)\sin x+\left(\sum_{\ell=1}^{n}{\displaystyle \dfrac{c_{\ell}}{x^{\ell}}}\right)\cos x=c_{0}
\]
where $x^{n+1}$ divides $\sin x$ rather than $\cos x.$

Section \ref{sec:Hyperbolic} shows how the inverses of modified spherical
Bessel function $i_{n}(x)$ can solve some analogous equations transformable
to the form
\[
\left(\sum_{\ell=1}^{n+1}\dfrac{s_{\ell}}{x^{\ell}}\right)\sinh x+\left(\sum_{\ell=1}^{n}\dfrac{c_{\ell}}{x^{\ell}}\right)\cosh x=c_{0}.
\]

Section \ref{sec:LambertW} shows how the inverses of modified spherical
Bessel function $k_{n}(x)$ can solve some analogous equations transformable
to the form
\[
\left(\sum_{\ell=1}^{n+1}\dfrac{c_{\ell}}{x^{\ell}}\right){\displaystyle e^{-x}}=c_{0},
\]
for which Lambert $W$ also solves the special case $n=0.$

\section{Inverses of $\boldsymbol{y_{n}(x)}$ solve some equations containing
$\boldsymbol{\cos x/x^{\ell}}$ and possibly also $\boldsymbol{\sin x/x^{\ell}}$\label{sec:SolvableViayn}}

As with most mathematics software, Mathematica implements principal-branch
inverses for all of its elementary functions, but implements inverses
for very few of its many special functions. Perhaps this is because
almost all mathematical functions in Mathematica work for nonreal
as well as real values of \textsl{all} their arguments, including
orders and other parameters that are usually real or integer; and
implementing inverse special functions for such multivariate complex
domains would be a daunting task.\footnote{Mathematica has a general purpose function named $\mathsf{InverseFunction}$,
but it offers no branch choice, and for versions through 12.1 it can
jump between branches or omit segments, as illustrated by executing,
for example
\[
\mathrm{\mathsf{Plot}}\,\left[\mathrm{\mathsf{InverseFunction}}\,[\mathrm{\mathsf{ExpIntegralEi]}}\,[y],\left\{ y,-4,1\right\} \right],
\]
which plots \textsl{nothing}.} However, AskConstants needs only real inverses that are much easier
to implement. Moreover, it is relatively easy to implement all or
at least many real branches for most special functions. Therefore
AskConstants includes such multi-branched real inverses for many of
the Mathematica special functions, including spherical Bessel functions
$y_{n}(x)$ and $j_{n}(x)$.
\begin{defn*}
A \textbf{strict univariate Laurent polynomial} is either 0 or a sum
of one or more terms that are a numeric coefficient times a \textsl{negative}
integer power of the variable.
\end{defn*}
This article is concerned with equations that can be transformed to
one of the forms
\begin{align*}
p(x)\,e^{-x} & =c_{0},\\
p(x)\cos x+q(x)\sin x & =c_{0},\\
p(x)\cosh x+q(x)\sinh x & =c_{0,}
\end{align*}
where $c_{0}$ is real and where $p(x)$ and $q(x)$ are strict Laurent
polynomials in variable $x$ with integer coefficients.

The spherical Bessel function of the second kind $y_{n}(x)$ is defined
as a specific solution to the ordinary differential equation
\[
x^{2}\dfrac{d^{2}y}{dx^{2}}+2x\dfrac{dy}{dx}+\left(x^{2}-n(n+1)\right)y=0.
\]

This function is less concisely representable as
\[
y_{n}(x)=\sqrt{\dfrac{\pi}{2x}}J_{n+\frac{1}{2}}(x),
\]
where $J_{n+\frac{1}{2}}(x)$ is an ordinary (cylindrical) Bessel
function of the second kind.

For integer orders $n$, the Mathematica $\mathsf{FunctionExpand}$
function transforms $y_{n}(x)$ into an exact closed form Laurent
cos sin representation partially listed in Table \ref{Table:SphericalBesselyn(z)ToTrigLarent-1}.
These table entries can be computed from a formula of Lord Rayleigh's
for $x\neq0$, known for over 100 years:
\begin{equation}
y_{n}(x):=-(-x)^{n}\left(\dfrac{1}{x}\dfrac{d}{dx}\right)^{n}\dfrac{\cos x}{x}\label{eq:RaleighForyn}
\end{equation}
or from $y_{0}(x)$, $y_{1}(x),$ and the recurrence
\begin{equation}
y_{n+1}(x):=\dfrac{2n+1}{x}y_{n}(x)-y_{n-1}(x).\label{eq:RecurrenceForynOrjn}
\end{equation}

\begin{table}[H]
\caption{Exact strict Laurent cos sin expansions of spherical Bessel functions
$y_{n}(x)$ }
\label{Table:SphericalBesselyn(z)ToTrigLarent-1}
\noindent \centering{}%
\begin{tabular}{|c|c|c|}
\hline 
$\!n\!$ & Partially expanded $y_{n}(x)=c_{0}$  & Solutions $\forall$ branches $b$ containing $c_{0}$\tabularnewline
\hline 
\hline 
0 & $-\dfrac{\cos x}{x}=c_{0}$ & $x=\mathrm{inverse}\,_{b}(y_{\boldsymbol{0}})(c_{0})$\rule[-14pt]{0pt}{34pt}\tabularnewline
\hline 
$1$ & $-\dfrac{\cos x}{x^{2}}-\dfrac{\sin x}{x}=c_{0}$ & $x=\mathrm{inverse}\,_{b}(y_{\boldsymbol{1}})(c_{0})$\rule[-14pt]{0pt}{34pt}\tabularnewline
\hline 
$\!2\!$ & $\left(-\dfrac{3}{x^{3}}+\dfrac{1}{x}\right)\cos z-\dfrac{3}{x^{2}}\sin x=c_{0}$ & $x=\mathrm{inverse}\,_{b}(y_{\boldsymbol{2}})(c_{0})$\rule[-14pt]{0pt}{34pt}\tabularnewline
\hline 
$\!3\!$ & $\!\left(\!-\dfrac{15}{x^{4}}+\dfrac{6}{x^{2}}\!\right)\cos x+\left(\!-\dfrac{15}{x^{3}}+\dfrac{1}{x}\!\right)\sin x=c_{0}\!$ & $x=\mathrm{inverse}\,_{b}(y_{\boldsymbol{3}})(c_{0})$\rule[-14pt]{0pt}{34pt}\tabularnewline
\hline 
$\vdots$ & $\vdots$ & $\vdots$\tabularnewline
\hline 
\end{tabular}
\end{table}

The strict Laurent polynomials in Table 1 are Bessel polynomials in
$1/x,$ which are a special case of Lommel polynomials \cite{DickinsonLommelPoly}.
Thus, for example, equations such as
\[
-\dfrac{\cos x}{x^{2}}-\dfrac{\sin x}{x}=c_{0}
\]
for constant $c_{0}$ in the range of the left side are solvable by
$\mathrm{inverse}\,_{b}(y_{1})(c_{0})$ for the branches $b$ that
contain $c_{0}.$ Published equations that can be transformed to this
form most often have \textsl{nonnegative} powers of $x$, making it
necessary to divide all the terms on both sides by the largest power
of $x$. Published equations for $n>0$ often have $\tan x$ or $\cot x$
rather than $\cos x$ and $\sin x$, requiring multiplication by $\cos x$
or $\sin x.$

Figure \ref{Figure:y0Thruy3AndOrdinateMinus1} superimposes plots
of $y_{n}(x)$ for order $n=0$ through $3$, with each order a different
color. As suggested by Table \ref{Table:SphericalBesselyn(z)ToTrigLarent-1}
and Figure \ref{Figure:y0Thruy3AndOrdinateMinus1}:
\begin{itemize}
\item $y_{n}(x)$ has odd symmetry for even $n$ and even symmetry for odd
$n$;
\item for even $n$, $y_{n}(x)=c_{0}$ always has at least one real solution
for all real $c_{0}$;
\item for odd $n$, $y_{n}(x)=c_{0}$ always has at least two real solutions
for all real $c_{0}$$\:\leq y_{n}(x_{n,1}),$ where $x_{n,1}$ is
the least positive abscissa for a local maximum of $y_{n}$;
\item a pole of order $n+1$ dominates as $x$$\,\shortrightarrow0$; 
\item $y_{n}(x)\shortrightarrow(\cos x)/(-x)^{n+1}$ for even $n$ or $y_{n}(x)\shortrightarrow(\sin x)/(-x)^{n}$
for odd $n$ as $x\shortrightarrow\pm\infty.$
\end{itemize}
\begin{figure}[H]
\caption{Spherical Bessel $y_{0}(x)$ through $y_{3}(x)$ and ordinate -1}

\label{Figure:y0Thruy3AndOrdinateMinus1}
\noindent \centering{}\includegraphics{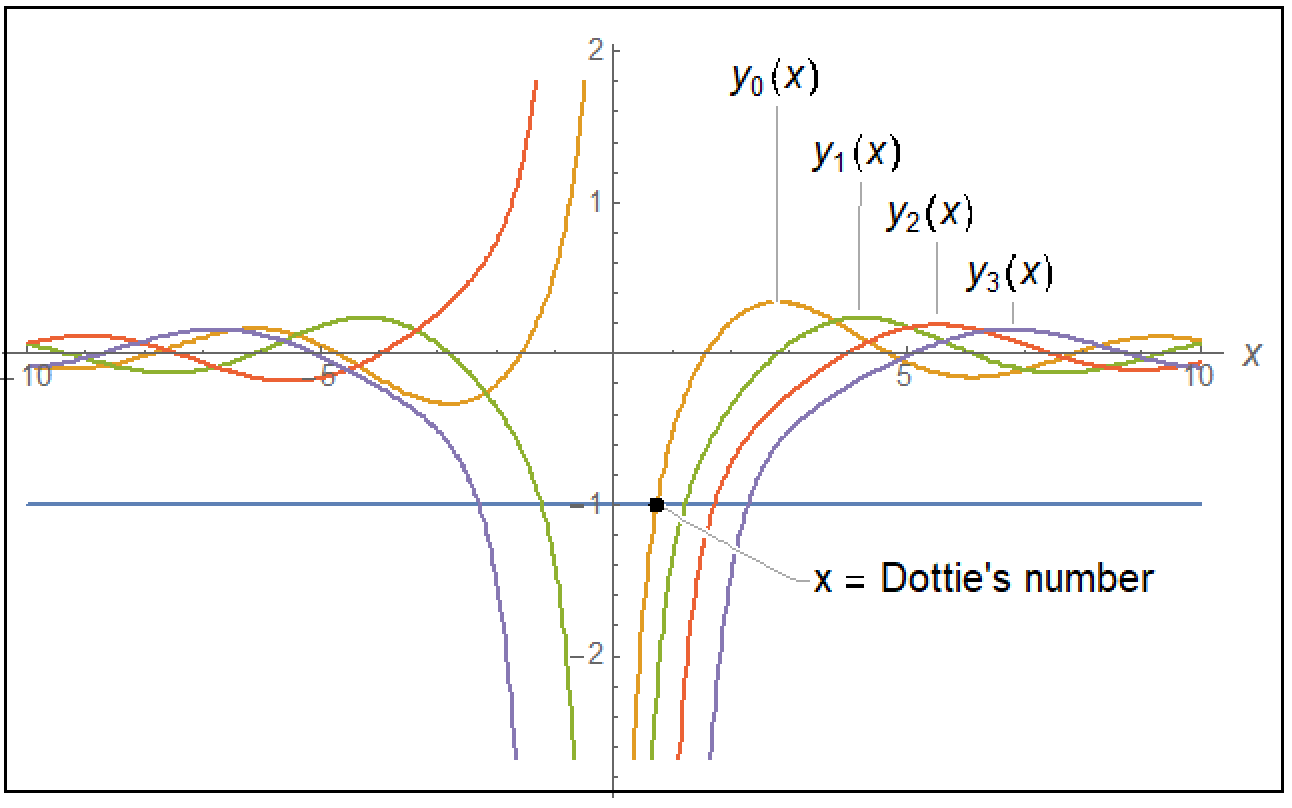}
\end{figure}
All of the real inverse functions for special functions $f(x)$ implemented
in AskConstants work by partitioning the portion of real $x$ where
$f(x)$ is real into maximally monotonic intervals terminated by the
refinable float local infima and suprema of $f$. Table \ref{Table:y0(x)InfimaAndSuprema}
lists six digits of the abscissas and ordinates of some of those interval
endpoints for $y_{0}(x)$, which AskConstants can compute for a particular
infimum or supremum number of $y_{n}(x)$ to precision $p$ by entering,
for example, 
\[
\mathrm{\mathsf{N}}\,[\mathrm{\mathsf{SphericalBesselYInfimumOrSupremumAbscissa}}\,[n,\mathit{infsupumNumber}],\,p]
\]
and
\[
\mathrm{\mathsf{N}}\,[\mathrm{\mathsf{SphericalBesselYInfimumOrSupremumOrdinate}}\,[n,\mathit{infsupumNumber}],\,p].
\]
The infimum or supremum having the least positive abscissa is number
1, with any to its right successively numbered 2, 3, \ldots{} and any
to its left successively numbered 0, -1, $\cdots$.

\begin{table}[H]
\begin{centering}
\caption{Some of the countably infinite local infima and suprema of $y_{0}(x)$}
\par\end{centering}
\begin{centering}
\label{Table:y0(x)InfimaAndSuprema}
\par\end{centering}
\noindent \centering{}%
\begin{tabular}{|c|c|l|}
\hline 
\begin{tabular}{c}
infsupum\tabularnewline
number $m$\tabularnewline
\end{tabular} & %
\begin{tabular}{c}
abscissa\tabularnewline
$x_{0,m}$\tabularnewline
\end{tabular} & %
\begin{tabular}{c}
ordinate\tabularnewline
$y_{0}\left(x_{0,m}\right)$\tabularnewline
\end{tabular}\tabularnewline
\hline 
\hline 
$\vdots$ & $\vdots$ & $\vdots$\tabularnewline
\hline 
-2 & -6.12125 & 0.161228\tabularnewline
\hline 
-1 & -2.79839 & -0.336508\tabularnewline
\hline 
0 & 0.00000 & $\begin{cases}
-\infty, & \mathrm{\mathsf{Direction}}\rightarrow\text{\textquotedblleft}\mathrm{\mathsf{FromAbove}}\text{\textquotedblright};\\
\infty, & \mathrm{\mathsf{Direction}}\rightarrow\text{\textquotedblleft}\mathrm{\mathsf{FromBelow}}\text{\textquotedblright}\\
\mathrm{\mathsf{ComplexInfinity}}, & \mathrm{otherwise}
\end{cases}$\tabularnewline
\hline 
1 & 2.79839 & 0.336508\tabularnewline
\hline 
2 & 6.12125 & -0.161228\tabularnewline
\hline 
$\vdots$ & $\vdots$ & $\vdots$\tabularnewline
\hline 
\end{tabular}
\end{table}
$\mathsf{SphericalBesselYInfimumOrSupremumOrdinate}$ has an optional
argument that appropriately returns either $\infty$ or $-\infty$
instead of $\mathsf{ComplexInfinity}$ when entered as $\mathsf{Direction}$
$\rightarrow$ ``$\mathsf{FromAbove}$'' or $\mathsf{Direction}$
$\rightarrow$ ``$\mathsf{FromBelow}$''.

Float values of nontrivial stationary infsupum abscissas of $f(x)$
are computed by using the iterative Mathematica function invocation
\[
\mathsf{FindRoot}\,[f'(x)=0,\mathit{guess},\mathit{lowerBound},\mathit{upperBound},\mathsf{WorkingPrecision}\rightarrow\ldots],
\]
using a guess that is the value correct to 16 significant digits.

Each maximally monotonic interval of $f(x)$ corresponds to a real
branch of the inverse function. Branch 1 is the branch whose right
endpoint is least positive. Branches for successively more positive
right endpoints are numbered 2, 3, etc., whereas branches for successively
smaller right endpoints are numbered 0, -1, etc. The left endpoint
belongs to the branch. The right endpoint belongs to the branch only
for the rightmost branch.

Float values of $\mathrm{inverse}\,_{b}(f)(c_{0})$ are calculated
by
\[
\mathsf{FindRoot}\,[f(x)=c_{0},\mathit{guess},\mathit{lowerBound},\mathit{upperBound},\mathsf{WorkingPrecision}\rightarrow\ldots]
\]
with bounds that are the bounding pair of adjacent infimum or supremum
abscissas. The guess is usually piecewise, using invertible truncated
series approximations to $f(x)$ at the end abscissas and for some
functions also at a point between. The guess is usually within 20\%
of the converged value, typically achieving convergence within about
5 to 20 iterations.\footnote{Rather than using $\mathsf{FindRoot}$, accelerated truncated reverted
series and other methods more specific to each implemented function
could be faster with more control over the resulting accuracy. However,
I wanted to implement quickly about 45 inverse functions, many of
which have parameters such as an order and/or several or a countably
infinite number of branches.}

Dividing both sides of Dottie's fixed-point equation (\ref{eq:DottiesFixedPointEqn})
by nonzero $-x$ then referring to Table \ref{Table:SphericalBesselyn(z)ToTrigLarent-1}
gives
\[
-1=-\dfrac{\cos x}{x}=y_{0}(x).
\]
Therefore Figure \ref{Figure:y0Thruy3AndOrdinateMinus1} also plots
a horizontal line through ordinate $-1$, which intersects $y_{0}(x)$
at and only at abscissa Dottie number $x\approxeq$ 0.73908513321516$\,.$
This number is between and only between infsupum abscissas $x_{0}=0$
and $x_{1}\approxeq2.79839$ in Table \ref{Table:y0(x)InfimaAndSuprema},
which delimit branch 1. Therefore the Dottie number is 
\[
\mathrm{inverse}\,_{\boldsymbol{1}}(y_{0})(-1).
\]

Figure \ref{Figure:RealInverseSphericalBesselY} superimposes plots
of a vertical line through abscissa $-1$ and the order 0 AskConstants
function $\mathrm{\mathsf{RealInverseSphericalBesselY\,}}[0,t,b]$
for branches $b=-2$ through branch 3, with each branch a different
color. Here branch 1 and only branch 1 of the inverse function intersects
the vertical line at the Dottie number $\approxeq$ 0.73908513321516$\,.$
\noindent \begin{center}
\begin{figure}[H]
\noindent \begin{centering}
\caption{$\,\mathrm{inverse}\,_{\boldsymbol{1}}(y_{0})(-1)=\mathrm{Dottie\:\:number}$ }
\label{Figure:RealInverseSphericalBesselY}
\par\end{centering}
\noindent \centering{}\includegraphics[scale=1.2]{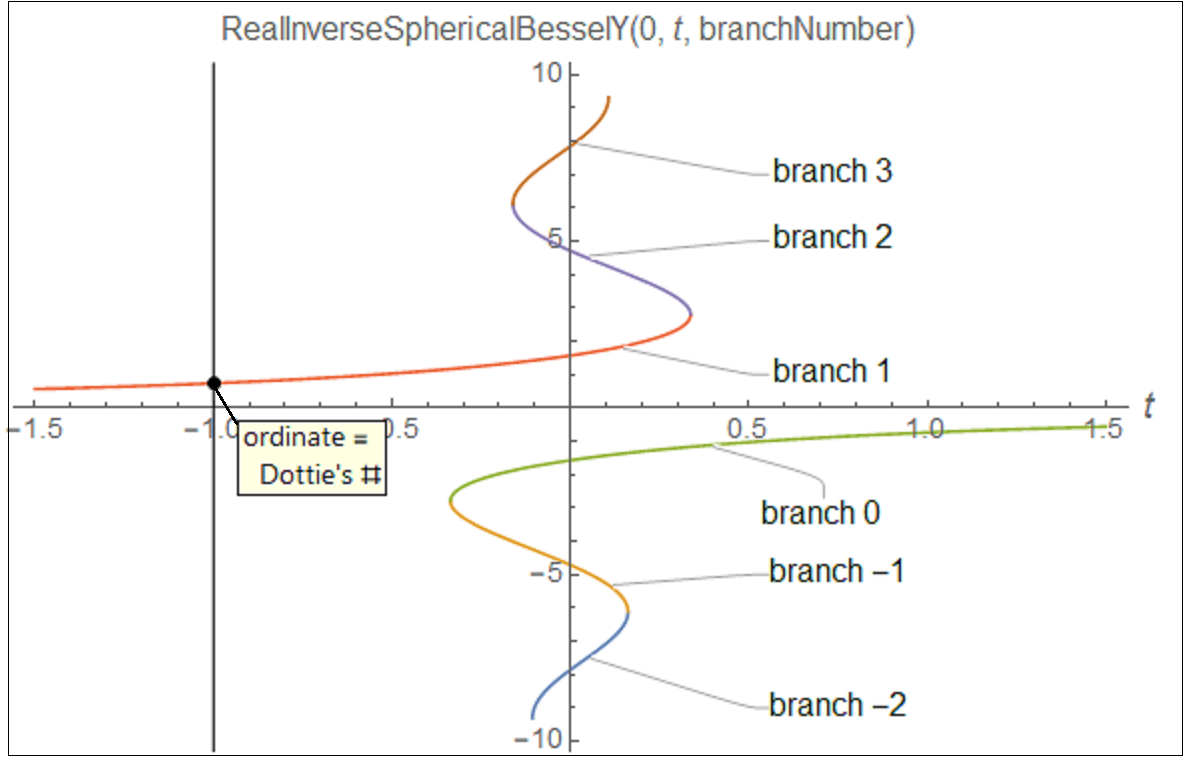}
\end{figure}
\par\end{center}

Figure \ref{Figure:LinesTangentToCosx} shows the four lines through
the origin tangent to $\cos x$ closest to the origin. The tangent
points are at abscissas and ordinates in Table \ref{Table:y0(x)InfimaAndSuprema}
except for the pole at infsupum $m=0\,.$ These tangent lines and
similar ones for larger $\left|m\right|$ suggest how the number of
real solutions to the \textbf{generalized Dottie equation}
\begin{equation}
\cos x=c_{0}\,x\label{eq:GeneralizedDottieEqn}
\end{equation}
changes as the value of $c_{0}$ decreases from $+\infty$: The one
intersection abscissa increases from $x=0^{+}$ through the Dottie
number until a negative solution also appears at the negative tangency
abscissa nearest the origin. Then that negative solution abscissa
bifurcates into two that increase in separation. Then another positive
solution appears at the positive tangency abscissa nearest the origin.
Then that solution bifurcates into two that increase in separation.
This process alternates at tangency points increasingly far from the
origin until $c_{0}=0$, making the solutions all be the zeros of
$\cos x$. Then pairs of points coalesce then disappear alternately
from the solutions furthest from the origin in the positive and negative
directions until there is just one solution at $x=0^{-}$.
\begin{figure}[H]

\caption{Some lines through the origin tangent to $\cos x$ at extrema of $\boldsymbol{y_{0}(x\protect\neq0)}$}

\noindent \label{Figure:LinesTangentToCosx}
\noindent \centering{}\includegraphics{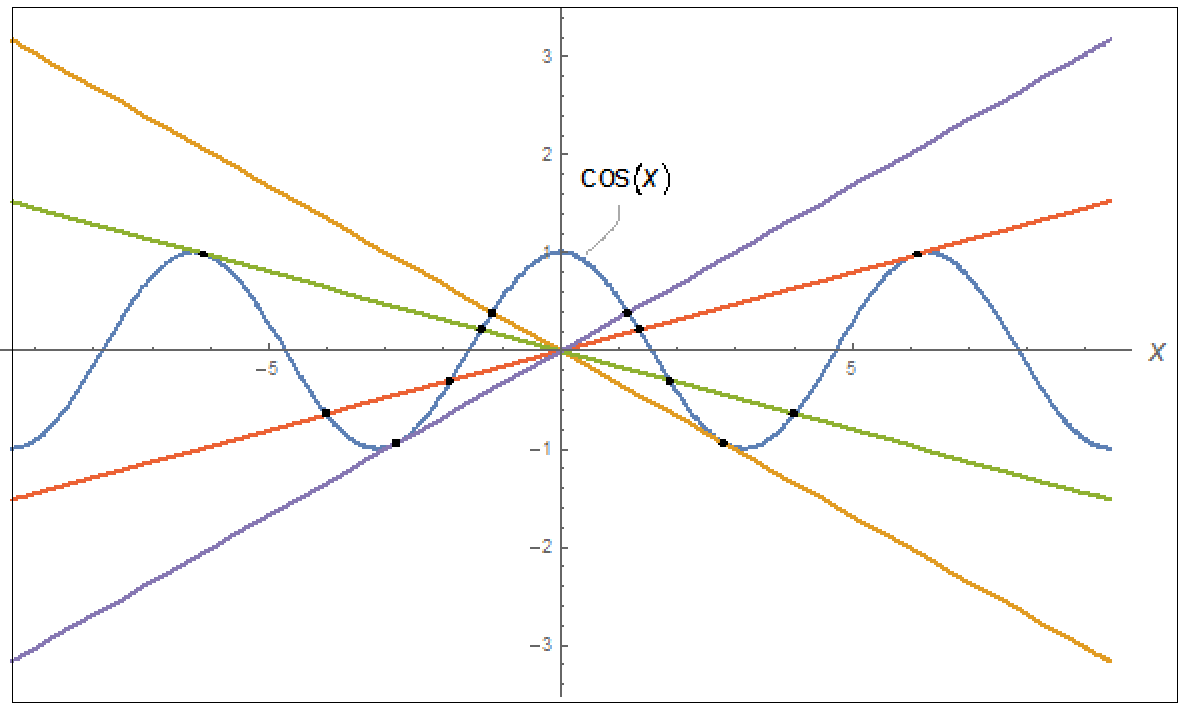}
\end{figure}

Constant recognition programs generate \textsl{conjectures}. It is
the responsibility of the user to prove or disprove them. For the
AskConstants conjecture that the Dottie number is exactly $\mathbf{inverse}\,_{1}(y_{0})(-1),$
it is easy to prove a much more general result:
\begin{prop}
\label{thm:yn}For real $x$ and an equation of the form
\begin{equation}
p(x)\cos x+q(x)\sin x=c_{0}\label{eq:ynOfxeqc0}
\end{equation}
from row $n$ of the countably infinite number of rows partially listed
in Table 2 that can be extended by recurrence $\mathrm{(}$\ref{eq:RecurrenceForynOrjn}$\mathrm{)}$,
let $x_{n,m}$ denote the $m^{\mathrm{th}}$ infsupum abscissa of
$the$ left side and let $B$ denote the set of all real branches
$b$ of $\mathrm{inverse}\,_{b}(y_{n})$ for which the interval
\[
[y_{n}(x_{n,\,b-1}),y_{n}(x_{n,b}))
\]
contains constant $c_{0}$.

Then all of the real solutions to equation $(\ref{eq:ynOfxeqc0})$,
if any, are
\[
x=\mathrm{inverse}\,_{b}(y_{n})(c_{0})
\]
where branches $b$ are all elements of set $B$.
\end{prop}

\begin{proof}
$\,$

$\text{\textopenbullet}$ The spherical Bessel functions $y_{n}(x)$
are exactly equivalent to the left sides of these equations by the
Raleigh formula (\ref{eq:RaleighForyn}) for $y_{n}(x)$.

$\text{\textopenbullet}$ $\mathrm{inverse}\,_{b}(y_{n})(c_{0})$
is defined as the one real solution to $y_{n}(x)=c_{0}$ in a maximally
monotonic continuous interval number $b$ of $x$ containing $c_{0}.$

$\text{\textopenbullet}$ The set $B$ of all such branches thus contains
all of the solutions.
\end{proof}
\begin{cor*}
An exact closed form for the Dottie number is $\mathrm{inverse}\,_{1}\left(y_{0}\right)\left(-1\right)$.
\end{cor*}
\begin{proof}
$\,$

$\text{\textopenbullet}$ Divide both sides of Dottie's equation (\ref{eq:DottiesFixedPointEqn})
by nonzero $-x$, then transpose the two sides giving the first equation
in Table \ref{Table:SphericalBesselyn(z)ToTrigLarent-1} with $c_{0}=-1$,
for which the possible solutions are $\mathrm{inverse}\,_{b}\left(y_{0}\right)\left(-1\right)$
by Proposition \ref{thm:yn}.

$\text{\textopenbullet}$ Referring to Table \ref{Table:y0(x)InfimaAndSuprema},
ordinate $-1$ occurs between and only between infsupum numbers 0
and 1, making the only branch $b=1,$ which includes the Dottie number
at abscissa $\approxeq0.739085$$\,.$
\end{proof}
\begin{xca*}
Determine both a float value and an exact explicit closed form for
the Dottie number in degree mode.
\end{xca*}
The initial reason for implementing the real inverse functions$\,$
$\mathrm{inverse}\,_{b}(f)$ $\,$was to apply them to the AskConstants
float input $\tilde{x}$ giving a transformed float $\tilde{t}$ to
help determine if the application could propose an exact constant
candidate $t$ for $\tilde{t}$. If so, then a candidate for the input
float is $f(t)$. However, as illustrated by this example, a bonus
of these real inverses is to enable recognition of them too, which
is often applicable to proposing exact candidates for float equation
solutions.

Generalized Dottie equations occur rather often in the literature.
For example, despite containing only about 2500 constants, Table 1
in Robinson and Potter \cite{RobinsonAndPotter} lists 22 such constants.

That table also lists eight constants containing $\tan x$ or $\cot x$
that can be converted to the form
\[
-\dfrac{\cos x}{x^{2}}-\dfrac{\sin x}{x}=y_{1}(x)=c_{0}
\]
and are therefore exactly solvable by $\mathrm{inverse}\,_{b}\left(y_{1}\right)$$\left(c_{0}\right)$.

The expansions illustrated in Table \ref{Table:SphericalBesselyn(z)ToTrigLarent-1}
are also valid in the complex domain. However, it might require substantial
effort to extend $\mathsf{RealInverseSphericalBesselY}$ to the complex
domain efficiently and robustly \textendash{} particularly if the
order is also generalized from nonnegative integers to the complex
domain. Meanwhile, Fettis \cite{ComplexRootsOfGeneralizedDottieEtAl}
lists numeric values of five complex solutions to the generalized
Dottie equation (\ref{eq:GeneralizedDottieEqn}) for each of about
30 values of $c_{0}.$

The general technique that enabled the exact solutions discussed in
this section and subsequent ones is to recognize when the left side
of an equation in the form
\[
\mathit{expression}=\mathit{constant}
\]
is an exact transformation of some named $f(x)$ for which we can
compute and concisely represent inverses.
\begin{rem*}
Gaidash \cite{GaidashDottie} derives a different explicit exact closed-form
solution to Dottie's equation:
\[
\arcsin\left(1-2I_{1/2}^{-1}\left(\dfrac{1}{2},\dfrac{3}{2}\right)\right)
\]
 where $I_{1/2}^{-1}$ is the inverse regularized incomplete beta
function.
\end{rem*}

\subsection{Extending computer algebra Solve functions}

Mathematica permits users to enhance the Solve or the more thorough
Reduce function with rewrite rules. Therefore you could enhance those
functions to transform suitable equations to strict Laurent $\cos$
$\sin$ form, then express their exact solutions using $\mathrm{inverse}\,_{b}\left(y_{n}\right).$
Or you could instead encourage the authors of your computer algebra
systems to implement those enhancements, including numeric $\mathrm{inverse}\,_{b}\left(y_{n}\right)(x)$.

\subsection{Fixed points of $\boldsymbol{f(x)}$ are also fixed points of $\boldsymbol{\mathrm{inverse}\,_{b}(f)(x)}$}

For any function $f(x)$, if $x=f(x)$ for a particular $x=x_{0}$,
then also $x_{0}=\mathrm{inverse}\,_{b}(f)(x_{0})$ for an appropriate
branch $b$, as proved in \cite{FixedPointOfAnInverse}. Thus, for
example, Table \ref{Table:SphericalBessejnlToTrigLarent} reveals
that the Dottie number, $\mathrm{inverse}\,_{1}\left(y_{0}\right)\left(-1\right),$
is also the real solution to the equation
\[
\arccos x=x,
\]
where $\arccos$ denotes the principal branch of the multi-branched
inverse cosine function.

\subsection{The stability of a fixed point is irrelevant}

The stability of a fixed point is irrelevant for this article, because
the AskConstants $\mathsf{RealInverse}$ functions use $\mathsf{FindRoot}$
rather than fixed point iteration, and $\mathsf{FindRoot}$ uses faster
more robust iterations.

\section{Inverses of $\boldsymbol{j_{n}(x)}$ solve some different equations
containing $\boldsymbol{\sin x/x^{\ell}}$ and possibly also $\boldsymbol{\cos x/x^{\ell}}$\label{sec:SolvableViajn}}

After proving then generalizing the AskConstants conjecture for the
Dottie number, I decided to investigate some analogous results for
some other inverse spherical Bessel functions. As illustrated by Table
\ref{Table:SphericalBessejnlToTrigLarent}, spherical Bessel $j_{n}(x)$
also has exact closed form representations similar to those for $y_{n}(x),$
except that $1/x^{n+1}$ multiplies $\sin x$ rather than $\cos x.$
The expansion for $n=0$ is also known as the sinc function. These
table entries can be computed for nonzero $x$ from another Rayleigh
formula for nonnegative integer $n$ that is supplemented here to
be correct also for $x=0$:
\[
j_{n}(x):=\begin{cases}
(-x)^{n}\left(\dfrac{1}{x}\dfrac{d}{dx}\right)^{n}\dfrac{\sin x}{x}, & x\neq0;\\
\begin{cases}
1, & n=0;\\
0, & n\neq0;
\end{cases} & x=0,
\end{cases}
\]
or from $j_{0}(x)$, $j_{1}(x)$, and the recurrence
\begin{equation}
j_{n+1}(x):=\dfrac{2n+1}{x}j_{n}(x)-j_{n-1}(x).\label{eq:RecurrenceForjn}
\end{equation}

\begin{table}[H]
\caption{Exact strict Laurent sin cos expansions of spherical Bessel functions
$j_{n}(x)$}
\label{Table:SphericalBessejnlToTrigLarent}
\noindent \centering{}%
\begin{tabular}{|c|c|c|}
\hline 
$\!\!n\!\!$ & Partially expanded $j_{n}(x)=c_{0}$ & $\!$Solutions $\forall$ branches $b$ containing $c_{0}$$\!$\tabularnewline
\hline 
\hline 
$\!\!0\!\!$ & $\begin{cases}
\dfrac{\sin x}{x}, & x\neq0;\\
1, & x=0\,
\end{cases}=c_{0}$ & $x=\mathrm{inverse}\,_{b}(j_{\boldsymbol{\mathbf{0}}})(c_{0})$\rule[-22pt]{0pt}{50pt}\tabularnewline
\hline 
$\!\!1\!\!$ & $\begin{cases}
\dfrac{\sin x}{x^{2}}-\dfrac{\cos x}{x}, & x\neq0;\\
0, & x=0\,
\end{cases}=c_{0}$ & $x=\mathrm{inverse}\,_{b}(j_{\boldsymbol{\boldsymbol{1}}})(c_{0})$\rule[-22pt]{0pt}{50pt}\tabularnewline
\hline 
$\!\!2\!\!$ & $\begin{cases}
\left(\dfrac{3}{x^{3}}-\dfrac{1}{x}\right)\sin x-\dfrac{3}{x^{2}}\cos x, & x\neq0;\\
0, & x=0\,
\end{cases}=c_{0}$ & $x=\mathrm{inverse}\,_{b}(j_{\boldsymbol{\boldsymbol{2}}})(c_{0})$\rule[-22pt]{0pt}{50pt}\tabularnewline
\hline 
$\!\!3\!\!$ & $\!\!\begin{cases}
\left(-\frac{15}{x^{4}}+\frac{6}{x^{2}}\right)\sin x+\left(-\frac{15}{x^{3}}+\frac{1}{x}\right)\cos x, & \!\!\!x\neq0;\\
0, & x=0\,
\end{cases}=c_{0}\!\!$ & $x=\mathrm{inverse}\,_{b}(j_{\boldsymbol{3}})(c_{0})$\rule[-18pt]{0pt}{44pt}\tabularnewline
\hline 
$\vdots$ & $\vdots$ & $\vdots$\tabularnewline
\hline 
\end{tabular}
\end{table}

Thus a procedure similar to that described in Section \ref{sec:SolvableViayn}
can determine exact closed-form solutions to some equations that are
transformable to one of the equations in the middle column of Table
\ref{Table:SphericalBessejnlToTrigLarent}.

As an example, OEIS constant A199460 \cite{OEIS} is the one \textsl{positive}
solution to
\begin{equation}
\sin x=\dfrac{1}{2}x,\label{eq:SinxEqxOn2}
\end{equation}
which is equivalent to the equation
\[
\dfrac{\sin x}{x}=j_{0}(x)=\dfrac{1}{2}
\]
for our \textsl{nonzero} $x$, making $x=\mathrm{inverse}\,_{b}(j_{0})(1/2)\:$$\approxeq1.89549$
for the appropriate branch $b$.

Figure \ref{Figure:sinxEqxOn2} plots the two sides of equation (\ref{eq:SinxEqxOn2}).
Figure \ref{Figure:j0Thruj3AndOrdinate1/2} plots $j_{0}(x)$ through
$j_{3}(x)$ with each order a different color, together with a horizontal
line through ordinate 1/2. Comparing these two figures, notice how
dividing both sides of equation (\ref{eq:SinxEqxOn2}) by $x$ annihilated
the unsought solution $x=0\,.$
\noindent \begin{center}
\begin{figure}[H]
\caption{Real solutions to $\sin x=x/2$}

\label{Figure:sinxEqxOn2}
\noindent \centering{}\includegraphics{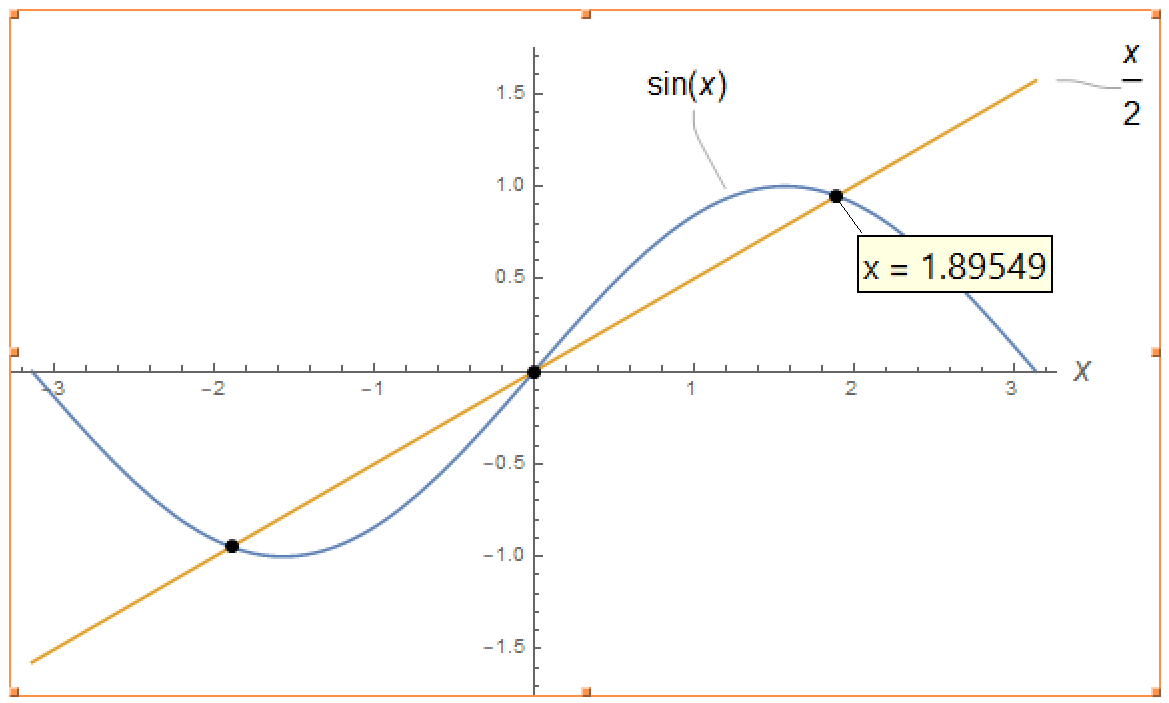}
\end{figure}
\par\end{center}

\noindent \begin{center}
\begin{figure}[h]
\caption{Spherical Bessel $j_{0}(x)$ through $j_{3}(x)$ and ordinate 1/2 }

\label{Figure:j0Thruj3AndOrdinate1/2}
\noindent \centering{}\includegraphics{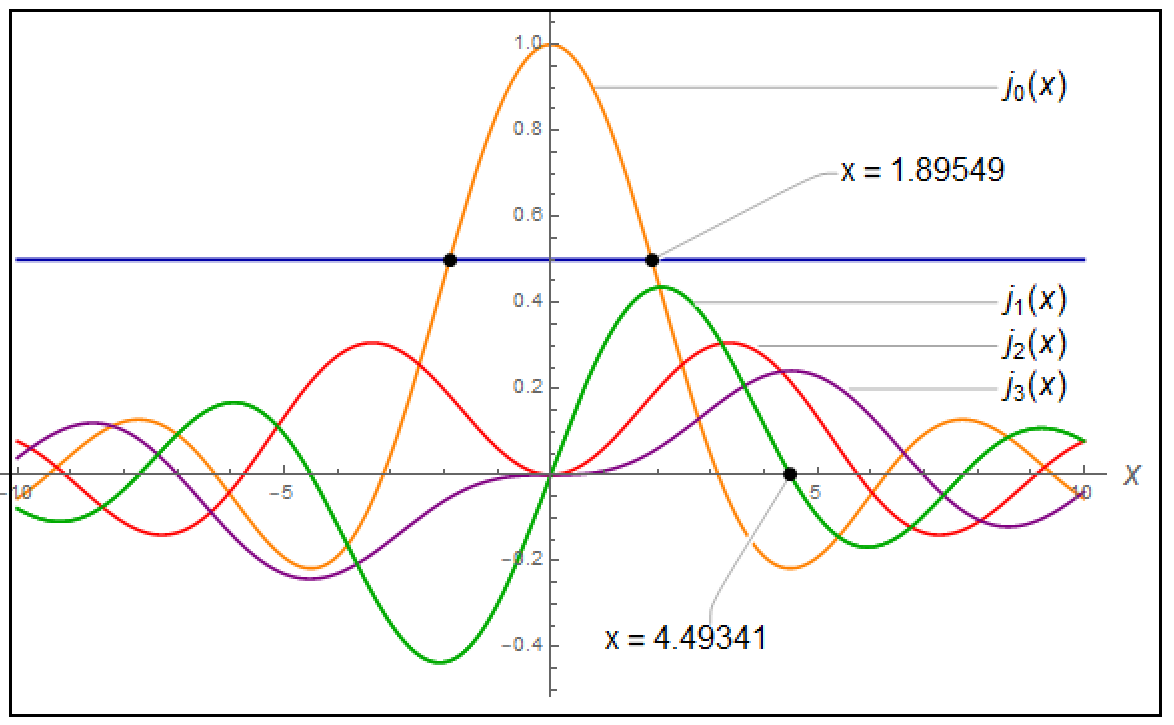}
\end{figure}
\par\end{center}

Table \ref{Table:InfimaAndSupremaOfj0(x)} lists six digits of the
abscissas and ordinates of some infima and suprema of $j_{0}(x)$.

\begin{table}[H]

\caption{Some of the countably infinite number of infima and suprema of $j_{0}(x)$ }

\label{Table:InfimaAndSupremaOfj0(x)}
\noindent \begin{centering}
\begin{tabular}{|c|c|c|}
\hline 
\begin{tabular}{c}
infsupum\tabularnewline
number $m$\tabularnewline
\end{tabular} & %
\begin{tabular}{c}
abscissa\tabularnewline
$x_{0,m}$\tabularnewline
\end{tabular} & %
\begin{tabular}{c}
ordinate\tabularnewline
$j_{0}\left(x_{0,m}\right)$\tabularnewline
\end{tabular}\tabularnewline
\hline 
\hline 
$\vdots$ & $\vdots$ & $\vdots$\tabularnewline
\hline 
-2 & -7.72525 & 0.128375\tabularnewline
\hline 
-1 & -4.49341 & -0.217234\tabularnewline
\hline 
0 & 0.00000 & 1.00000\tabularnewline
\hline 
1 & 4.49341 & -0.217234\tabularnewline
\hline 
2 & 7.72525 & 0.128375\tabularnewline
\hline 
$\vdots$ & $\vdots$ & $\vdots$\tabularnewline
\hline 
\end{tabular}
\par\end{centering}
\end{table}

Figure \ref{Figure:RealInverseSphericalBesselJ =00005B0,1/2,1=00005D}
superimposes plots of a vertical line through abscissa $1/2$ and
the order zero AskConstants function $\mathrm{\mathsf{RealInverseSphericalBesselJ}\,}[0,t,b]$
for branches $b=-2$ through 3, with each branch a different color.
Here branches 0 and 1 of the inverse function intersect the vertical
line through abscissa $t=1/2$ only at ordinates $\mathrm{inverse}\,_{\boldsymbol{0}}\left(j_{0}\right)\left(1/2\right)\approxeq-1.89549$
and $\mathrm{inverse}\boldsymbol{\,}_{\boldsymbol{1}}\left(j_{0}\right)\left(1/2\right)\approxeq1.89549$$\,.$

\begin{figure}[H]

\caption{$\mathsf{RealInverseSphericalBesselJ}$$\,[0,1/2,1]$ solves $\sin x=\frac{1}{2}x$}

\label{Figure:RealInverseSphericalBesselJ =00005B0,1/2,1=00005D}
\noindent \centering{}\includegraphics{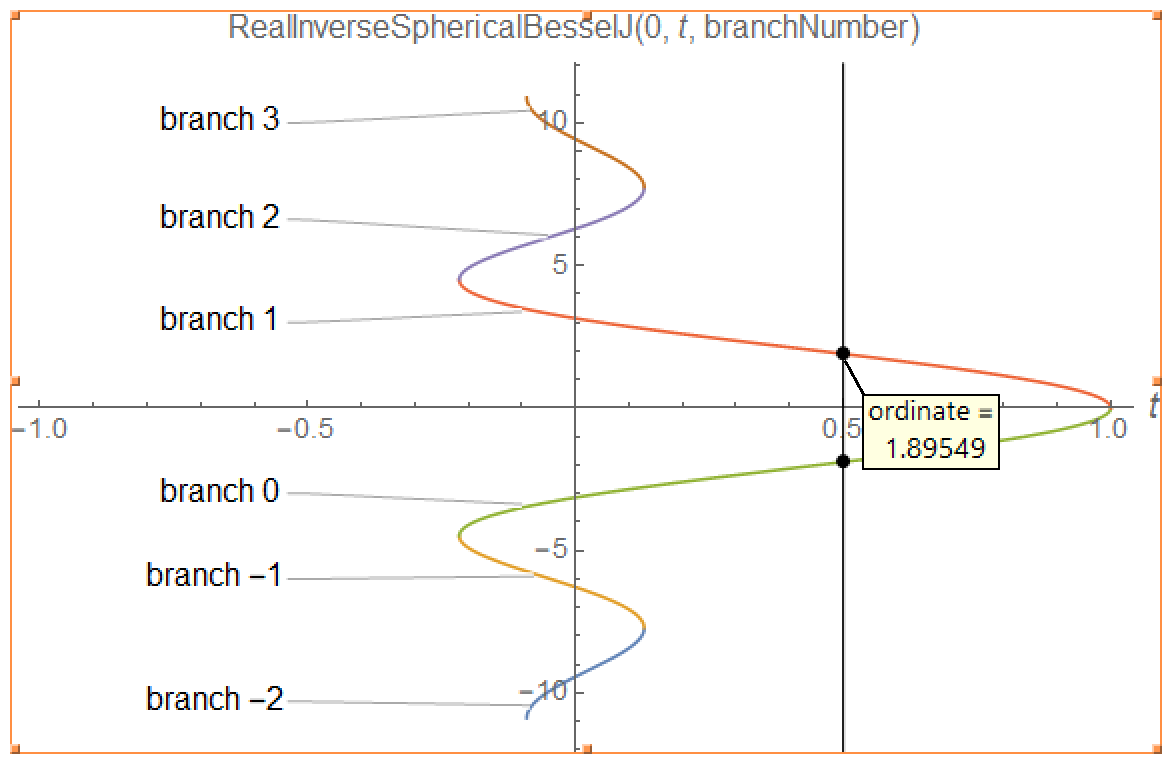}
\end{figure}

Figure \ref{Figure:PlotLInesTangentToSinx} shows the three lines
through the origin tangent to $\sin x$ closest to the origin. The
tangent points are at abscissas $x_{0,m}$ and ordinates $\sin x_{0,m}$
in Table \ref{Table:InfimaAndSupremaOfj0(x)}. These tangent lines
and similar ones for larger $\left|m\right|$ suggest how the number
of real solutions to the generalized equation
\begin{equation}
\sin x=c_{0}\,x\label{eq:SinxEqc0x}
\end{equation}
increases from 1 to 3, then 5, etc. as the absolute value of $c_{0}$
decreases from $+\infty$ to 0 where there are a countably infinite
number of solutions, then decreases to 1 as $c_{0}$ further decreases
toward $-\infty$.

Fettis \cite{ComplexRootsOfGeneralizedDottieEtAl} lists numeric values
of five complex solutions to the generalized equation (\ref{eq:SinxEqc0x})
for about 30 values of $c_{0}.$

\begin{figure}[H]
\caption{Some lines through the origin tangent to $\sin x$}

\label{Figure:PlotLInesTangentToSinx}
\noindent \centering{}\includegraphics{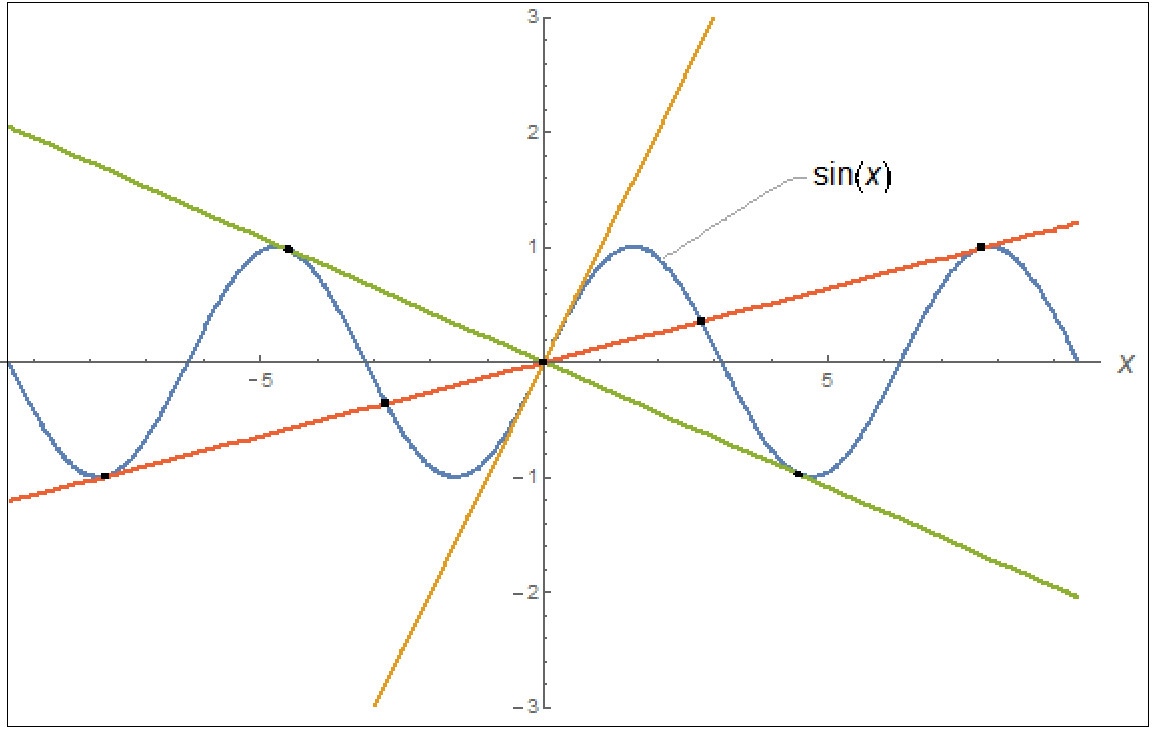}
\end{figure}

To summarize for spherical Bessel $j_{n}(x)$:
\begin{prop}
For real $x$ and an equation of the form
\begin{equation}
p(x)\sin x+q(x)\cos x=c_{0}\label{eq:jnOfxeqc0}
\end{equation}
from row $n$ of the countably infinite number of rows partially listed
in Table \ref{Table:SphericalBessejnlToTrigLarent} that can be extended
by recurrence $\mathrm{(}$\ref{eq:RecurrenceForjn}$\mathrm{)}$,
let $x_{n,m}$ denote the $m^{\mathrm{th}}$ infsupum abscissa of
$the$ left side and let $B$ denote the set of all real branches
$b$ of $\mathrm{inverse}\,_{b}(y_{n})$ for which the interval
\[
[j_{n}(x_{n,\,b-1}),j_{n}(x_{n,b}))
\]
contains constant $c_{0}.$

Then all of the real solutions to equation $(\ref{eq:jnOfxeqc0})$,
if any, are
\[
x=\mathrm{inverse}\,_{b}(j_{n})(c_{0})
\]
where branches $b$ are all elements of set $B$.
\end{prop}

\begin{proof}
$\,$Similar to that of Proposition \ref{thm:yn}.
\end{proof}
As an example for higher-order Bessel $j_{n}$, OEIS \cite[A115365]{OEIS}
lists the float value $4.49341$ of the least positive solution to
\begin{equation}
\tan x-x=0,\label{eq:A115365TanxEq2x}
\end{equation}
which is the abscissa of the least positive local minimum of $j_{0}(x)=\mathrm{sinc}\,(x)$.

Multiplying both sides of equation (\ref{eq:A115365TanxEq2x}) by
$(\cos x)/x^{2}$ gives for $x\neq0$
\begin{equation}
\dfrac{\sin x}{x^{2}}-\dfrac{\cos x}{x}=j_{1}(x)=0\label{eq:j1(x) =00003D 0}
\end{equation}
from Table 3. Figure \ref{Figure:j0Thruj3AndOrdinate1/2} reveals
that the least positive solution to equation (\ref{eq:j1(x) =00003D 0})
is in branch $b=2$ of $\mathrm{inverse}\,_{b}(j_{1})(0)$, making
$x=\mathrm{inverse}\,_{2}(j_{1})(0)\approxeq4.49341$. Because of
the identity
\[
j_{n}(z)=\sqrt{\dfrac{\pi}{2z}}J_{n+1/2}(z)
\]
and $c_{0}=0,$ Vladimir Reshetnikov \cite[A115365]{OEIS} could express
the solution as the Mathematica function invocation $\mathsf{BesselJZero}\,[3/2,1]$.

\subsection{Kepler's equation for elliptic orbits is a near miss}
\begin{flushright}
``\textsl{Kepler's equation is central to orbital mechanics}.''\\
\textendash{} me, being whimsical.
\par\end{flushright}

A common form of Kepler's equation for elliptic orbits is
\begin{equation}
M=E-e\sin E\label{eq:KeplersElliptical}
\end{equation}
where $0\leq e<1$ is the \textsl{eccentricity}, $-\pi<M\leq\pi$
is the mean anomaly, and $E$ is the eccentric anomaly. This equation
is widely regarded as having no currently known explicit exact closed
form solution when $E$ is the unknown, although a few authors regard
solutions containing raw integrals or infinite series as closed forms.

Rearranging terms and factors of equation (\ref{eq:KeplersElliptical})
into
\[
\sin E=\dfrac{1}{e}E-\dfrac{M}{e},
\]
which matches the related equation form for $n=0$ in Table \ref{Table:SphericalBessejnlToTrigLarent}
with slope $c_{0}=1/e$ only when $M=0\,.$ However, $1/e\geq1,$
and Figure \ref{Figure:PlotLInesTangentToSinx} illustrates that the
only solution is then $E=0$$\,.$

But it makes me wonder if there might be an infinite series of the
form
\[
E=M+\sum_{n=1}^{\infty}a_{m}\,j_{n}(ne)\sin(nM)
\]
with appropriate coefficients $a_{m}$ that has better properties
than the well known series representation
\[
E=M+\sum_{n=1}^{\infty}\dfrac{2}{n}J_{n}(ne)\sin(nM)
\]
for $e<1$ and $-\pi\leq M\leq\pi$, where $J_{n}$ is the ordinary
Bessel $J$ function. Spherical Bessel functions seem more appropriate
than the ordinary (cylindrical) Bessel functions for the inverse square
law of Newtonian gravity.

\section{Inverses of $\boldsymbol{i_{n}(x)}$ solve some equations containing
$\boldsymbol{\sinh x/x^{\ell}}$ and possibly also $\boldsymbol{\cosh x/x^{\ell}}$\label{sec:Hyperbolic}}

Inverses of the modified spherical Bessel function of the first kind,
$i_{n}(x)$, can solve equations analogous to those of Section \ref{sec:SolvableViajn}
but containing $\sinh$ and $\cosh$ rather than $\sin$ and $\cos$.
These functions have analogous exact expansions partially displayed
in Table \ref{Table:SphericalBesselinToSinhCoshLaurent}. These table
entries can be computed for $x\neq0$ and nonnegative $n$ from a
Rayleigh-type formula that is supplemented here to be correct also
for $x=0:$
\[
i_{n}(x):=\begin{cases}
x^{n}\left(\dfrac{1}{x}\dfrac{d}{dx}\right)^{n}\dfrac{\sinh x}{x}, & x\neq0;\\
\begin{cases}
1, & n=0;\\
0, & n>0,
\end{cases} & x=0;
\end{cases}
\]
or from $i_{0}(x)$, $i_{1}(x)$, and the recurrence
\begin{equation}
i_{n+1}(x):=i_{n-1}(x)-\dfrac{2n+1}{x}i_{n}(x).\label{eq:Recurrencein}
\end{equation}

\noindent \begin{center}
\begin{table}[h]
\caption{Exact strict Laurent sinh cosh expansions of spherical Bessel functions
$i_{n}(z)$\rule[-1pt]{0pt}{14pt}}
\label{Table:SphericalBesselinToSinhCoshLaurent}
\noindent \centering{}%
\begin{tabular}{|c|c|c|}
\hline 
$\!$$\!n\!\!$ & Partially expanded $i_{n}(x)=c_{0}$ & $\!\!$Solutions $\forall$ branches $b$ containing $c_{0}$$\!$$\!$\rule[-1pt]{0pt}{14pt}\tabularnewline
\hline 
\hline 
$\!\!0\!\!$ & $\begin{cases}
\dfrac{\sinh x}{x}, & x\neq0;\\
1, & x=0\,
\end{cases}=c_{0}$ & $x=\mathrm{inverse}\,_{b}(i_{\boldsymbol{\boldsymbol{0}}})(c_{0})$\rule[-20pt]{0pt}{47pt}\tabularnewline
\hline 
$\!\!1\!\!$ & $\begin{cases}
-\dfrac{\sinh x}{x^{2}}+\dfrac{\cosh x}{x}, & x\neq0;\\
0, & x=0
\end{cases}=c_{0}$ & $x=\mathrm{inverse}\,_{b}(i_{\boldsymbol{\boldsymbol{1}}})(c_{0})$\rule[-19pt]{0pt}{46pt}\tabularnewline
\hline 
$\!\!2\!\!$ & $\begin{cases}
\left(\dfrac{3}{x^{3}}+\dfrac{1}{x}\right)\sinh x-\dfrac{3}{x^{2}}\cosh x, & x\neq0;\\
0, & x=0
\end{cases}=c_{0}$ & $x=\mathrm{inverse}\,_{b}(i_{\boldsymbol{\boldsymbol{2}}})(c_{0})$\rule[-13pt]{0pt}{42pt}\tabularnewline
\hline 
$\!\!3\!\!$ & $\!\!\begin{cases}
\left(\!-\frac{15}{x^{4}}-\frac{6}{x^{2}}\!\right)\sinh x+\left(\!-\frac{15}{x^{3}}+\frac{1}{x}\!\right)\cosh x,\!\!\! & x\neq0;\!\\
0, & \:x=0\,\,
\end{cases}\!=c_{0}\!\!$ & $\!$$\begin{array}{c}
\!\!x=\mathrm{inverse}\,_{b}(i_{\boldsymbol{\boldsymbol{3}}})(c_{0})\!\end{array}\!$\rule[-17pt]{0pt}{40pt}\tabularnewline
\hline 
$\vdots$ & $\vdots$ & $\vdots$\tabularnewline
\hline 
\end{tabular}
\end{table}
\par\end{center}

Mathematica through version 12.1 has no built-in spherical Bessel
function $i_{n}$ or its inverses, but I have implemented them for
inclusion in the next version of AskConstants after 5.0, using the
techniques described in Section \ref{sec:SolvableViayn}.

Figure \ref{Figure:i0Thrui3AndPi} shows a plot of spherical Bessel
$i_{0}(x)$ through $i_{3}(x)$.

\begin{figure}[h]
\caption{A plot of $\pi$ and spherical Bessel $i_{0}\left(x\right)$ through
$i_{3}(x)$}
\label{Figure:i0Thrui3AndPi}
\noindent \centering{}\includegraphics[viewport=0bp 0bp 293.9811bp 189.25bp]{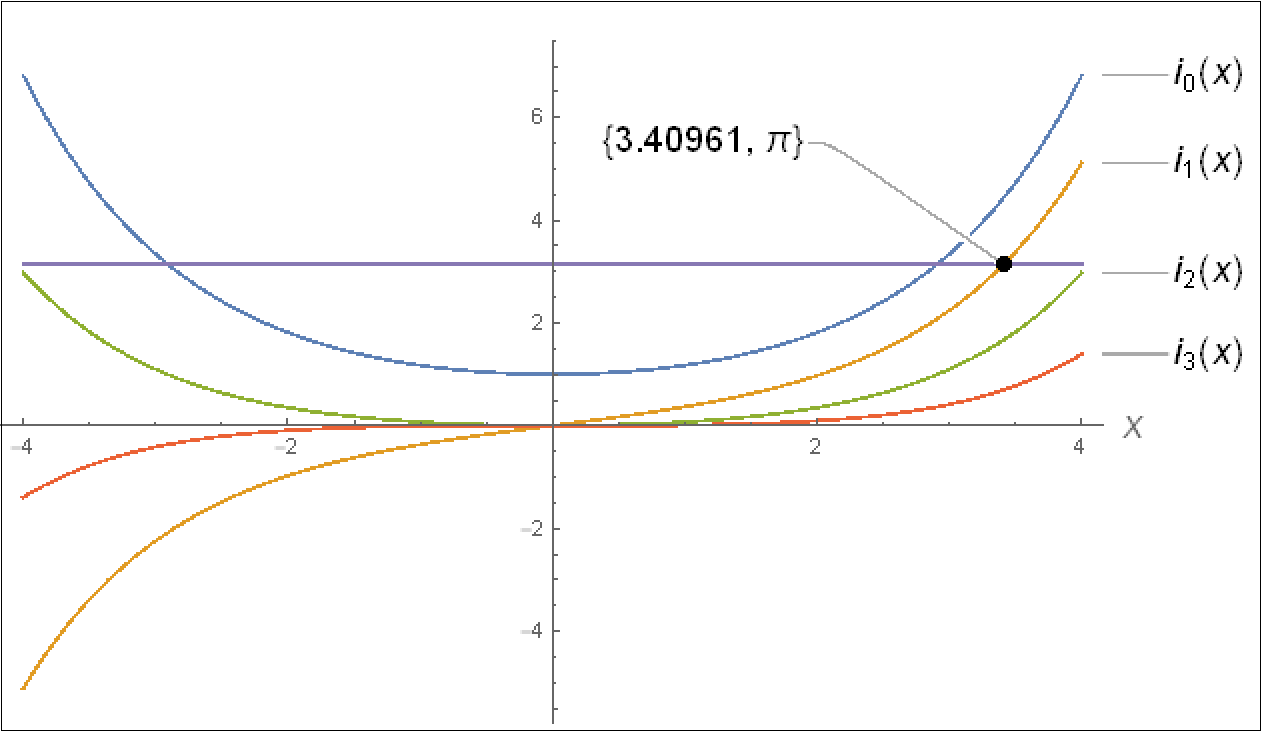}
\end{figure}
As suggested by Figure \ref{Figure:i0Thrui3AndPi} and Table \ref{Table:SphericalBesselinToSinhCoshLaurent}:
\begin{enumerate}
\item The even-order functions have even symmetry and the odd-order functions
have odd symmetry.
\item The even order functions only have real branches numbered 0 and 1.
\item The only real branch of odd order functions is branch 1.
\item For real $c_{0},$ any equation transformable to a form partially
listed in the rightmost column of Table \ref{Table:SphericalBesselinToSinhCoshLaurent}
has
\end{enumerate}
\begin{itemize}
\item one real solution for odd $n,$
\item one solution for $n=0$ and $c_{0}=1$ or two solutions for $n=0$
and $c_{0}>1\,,$
\item one solution for even $n\geq2$ and $c_{0}=0$ or two solutions for
even $n\geq2$ and $c_{0}>0\,.$
\end{itemize}
To summarize for spherical Bessel $i_{n}(x):$
\begin{prop}
For real $x$ and an equation of the form
\begin{equation}
p(x)\sinh x+q(x)\cosh x=c_{0}\label{eq:inOfxeqc0}
\end{equation}
from row $n$ of the countably infinite number of rows partially listed
in Table \ref{Table:SphericalBesselinToSinhCoshLaurent} that can
be extended by recurrence $\mathrm{(}$\ref{eq:Recurrencein}$\mathrm{)}$,
let $x_{n,m}$ denote the $m^{\mathrm{th}}$ infsupum abscissa of
$the,$left side and let $B$ denote the set of all real branches
b of $\mathrm{inverse}\,_{b}(y_{n})$ for which the interval
\[
[i_{n}(x_{n,\,b-1}),i_{n}(x_{n,b}))
\]
contains $c_{0}.$

Then all of the real solutions to equation $(\ref{eq:inOfxeqc0})$,
if any, are
\begin{equation}
x=\mathrm{inverse}\,_{b}(i_{n})(c_{0})\label{eq:xEqInvbOfinOfc0}
\end{equation}
where branches $b$ are all elements of set $B$.
\end{prop}

\begin{proof}
$\,$Similar to that of Proposition \ref{thm:yn}.
\end{proof}
I have not found an example of such equations published in paper or
on-line form, refereed or not. I welcome published examples and unpublished
ones that have a natural application.

Because a fixed point of $f(x)$ is a fixed point of $\mathrm{inverse}\,_{b}(f)$,
$\mathrm{inverse}\,_{1}(i_{0})(c_{0})$ can also solve equations transformable
to the form
\[
\dfrac{\mathrm{arcsinh}\,x}{x}=c_{0},
\]
but I have not found a published examples for equations equivalent
to that either.

\subsection{Kepler's equation for hyperbolic paths is also a near miss}

A common form of Kepler's equation for hyperbolic paths is
\begin{equation}
M=e\sinh H-H\label{eq:KeplersHyperbolic}
\end{equation}
where $e>1$ is the \textsl{eccentricity}, $M$ is the mean anomaly,
and $H$ is the hyperbolic eccentric anomaly. This equation is also
widely regarded as currently having no known explicit closed form
solution when $H$ is the unknown.

The terms and factors of equation (\ref{eq:KeplersHyperbolic}) can
be rearranged into
\[
\sinh H=\dfrac{1}{e}H+\dfrac{M}{e},
\]
which matches the related equation form for $n=0$ in Table \ref{Table:SphericalBesselinToSinhCoshLaurent}
with $c_{0}=1/e$ only when $M=0\,.$ However, line slope $1/e<1$
for hyperbolic paths, and Figure \ref{Figure:SinhxAnd3Slopes} illustrates
that the only solution is then $H=0$$\,.$

\begin{figure}[H]
\caption{$\mathrm{inverse}\,(i_{0})$ does not generally solve the hyperbolic
Kepler equation }

\label{Figure:SinhxAnd3Slopes}
\noindent \centering{}\includegraphics{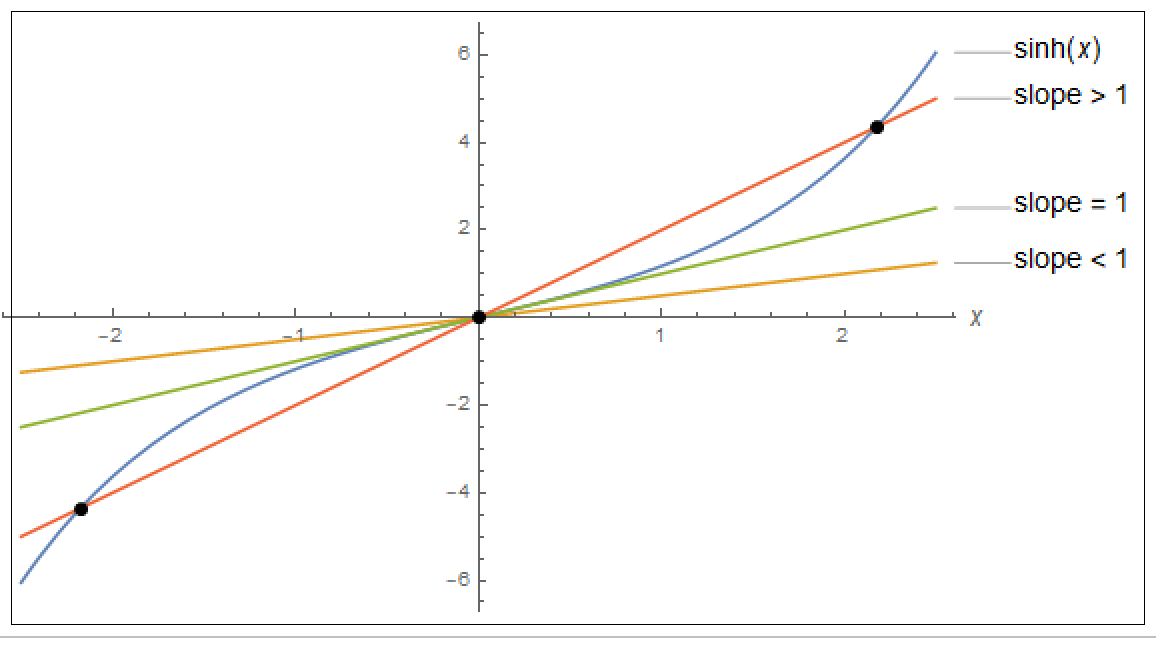}
\end{figure}

But it makes me wonder if there might be a useful infinite series
of the form
\[
H=M+\sum_{n=1}^{\infty}a_{m}i_{n}(ne)\sin(nM)
\]
with appropriate coefficients $a_{m}$.

\section{inverse$_{\boldsymbol{b}}\boldsymbol{\left(k_{n}\right)}$ generalizes
Lambert $W$\label{sec:LambertW}}

Some equations transformable to the form $\left({\textstyle \sum_{\ell=1}^{n}}\dfrac{c_{\ell}}{x^{\ell}}\right)e^{-x}=c_{0}$
can be solved by the inverse spherical Bessel function of the second
kind, inverse$_{b}\left(k_{n}\right)$$\left(c_{0}\right)$, where
$c_{\ell}$ are certain integer constants and $c_{0}$ is any nonzero
constant in the range of $k_{n}(x).$

For nonnegative integer $n$, $k_{n}(x)$ can be expressed in strict
Laurent negative exponential form, as listed in Table \ref{Table:SphericalBesselToExponentialLarent},
which is adapted from Weisstein \cite{WeissteinModifiedSphericalBesselFunctions2ndKind},
who also lists a recurrence for the rational factors $q_{n}(x)$ multiplying
$e^{-x}$:
\begin{equation}
q_{n}(x)=q_{n-1}(x)+\dfrac{2n-1}{x}q_{n-2},\label{eq:RecurrenceForkn}
\end{equation}
which can be used to extend the table, computing $k_{n}(x)$ for larger
$n$ from $k_{0}(x)$ and $k_{1}(x)$.

Alternatively, a Rayleigh-type formula is
\[
k_{n}(x)=(-x)^{n}\left(\dfrac{1}{x}\dfrac{d}{dx}\right)^{n}\dfrac{e^{-x}}{x}.
\]
\begin{table}[H]
\caption{Exact exponential expansions of spherical Bessel functions $k_{n}(x)$}
\label{Table:SphericalBesselToExponentialLarent}
\noindent \centering{}%
\begin{tabular}{|c|c|c|}
\hline 
$n$ & Partially expanded $k_{n}(x)=c_{0}$ & Solutions $\forall\:\mathrm{branches}$$\:b\:\mathrm{containing}\:c_{0}$\rule[-1pt]{0pt}{14pt}\tabularnewline
\hline 
\hline 
0 & $\dfrac{1}{x}e^{-x}=c_{0}$ & $x=\mathrm{inverse}\,_{b}(k_{0})(c_{0})$\rule[-11pt]{0pt}{30pt}\tabularnewline
\hline 
1 & $\left(\dfrac{1}{x^{2}}+\dfrac{1}{x}\right)e^{-x}=c_{0}$ & $x=\mathrm{inverse}\,_{b}(k_{1})(c_{0})$\rule[-13pt]{0pt}{33pt}\tabularnewline
\hline 
2 & $\left(\dfrac{3}{x^{3}}+\dfrac{3}{x^{2}}+\dfrac{1}{x}\right)e^{-x}=c_{0}$ & $x=\mathrm{inverse}\,_{b}(k_{2})(c_{0})$\rule[-13pt]{0pt}{33pt}\tabularnewline
\hline 
3 & $\left(\dfrac{15}{x^{4}}+\dfrac{15}{x^{3}}+\dfrac{6}{x^{2}}+\dfrac{1}{x}\right)e^{-x}=c_{0}$ & $x=\mathrm{inverse}\,_{b}(k_{3})(c_{0})$\rule[-13pt]{0pt}{33pt}\tabularnewline
\hline 
4 & $\left(\dfrac{105}{x^{5}}+\dfrac{105}{x^{4}}+\dfrac{45}{x^{3}}+\dfrac{10}{x^{2}}+\dfrac{1}{x}\right)e^{-x}=c_{0}$ & $x=\mathrm{inverse}\,_{b}(k_{4})(c_{0})$\rule[-13pt]{0pt}{33pt}\tabularnewline
\hline 
$\vdots$ & $\vdots$ & $\vdots$\tabularnewline
\hline 
\end{tabular}
\end{table}

If a given equation contains some positive powers of the unknown $x$,
then divide by the largest such power of $x$. If a given strict Laurent
polynomial multiplies $e^{x}$ rather than $e^{-x},$ then the substitution
$x\mapsto-s$ transforms that product to a strict Laurent negative
exponential form.

Beware that the Digital Library of Mathematical Functions \cite{NIST DLMF}
defines $k_{n}(x)$ as $\pi/2$ times these definitions. I am interested
in opinions about if I should use that definition instead.
\begin{flushright}
``\textsl{The good thing about standards is that there are so many
to choose from}.''\\
\textendash{} Andrew S. Tanenbaum
\par\end{flushright}

Mathematica through version 12.1 has no builtin spherical Bessel function
$k_{n}$, but the AskConstants application has one implemented one
using these expansions, and AskConstants also has a function
\[
\mathrm{\mathsf{RealInverseSphericalBesselK}}\,[n,t,\mathit{branch}]
\]
implemented using the techniques described in Section \ref{sec:SolvableViayn}.

Figure \ref{Figure:Plotk0thruk3} plots $k_{0}(x)$ through $k_{3}(x)$
with each order a different color. From that Figure and Table \ref{Table:SphericalBesselToExponentialLarent},
notice how
\begin{itemize}
\item a pole of order $n+1$ dominates as $x$$\,\shortrightarrow0$; 
\item $k_{n}(x)\sim e^{-x}/x$ for large $\left|x\right|$, making $k_{n}(x)\rightarrow0$
as $x\shortrightarrow+\infty$ and $k_{n}(x)\rightarrow-\infty$ as
$x\shortrightarrow-\infty;$
\item $k_{n}(x)$ has only branches 0 and 1 for odd $n;$
\item $k_{n}(x)$ has only branches $-1$, 0 and 1 for even $n;$
\item For odd $n,$ $k_{n}(x)=c_{0}$ has two solutions when $c_{0}>0$
or one solution when $c_{0}\leq0;$
\item For even $n$, $k_{n}(x)=c_{0}$ has one solution when $c_{0}>0$
or when $c_{0}=k_{n}(x_{n,\,-1})$, versus two solutions when $c_{0}<k_{n}(x_{n,\,-1})$,
where $x_{n,\,-1}$ is the abscissa of the local maximum at $\mathrm{infsupum}\,_{-1}(k_{n})$.
\end{itemize}
\begin{figure}[H]

\caption{Modified spherical Bessel function of the second kind}

\label{Figure:Plotk0thruk3}
\noindent \centering{}\includegraphics{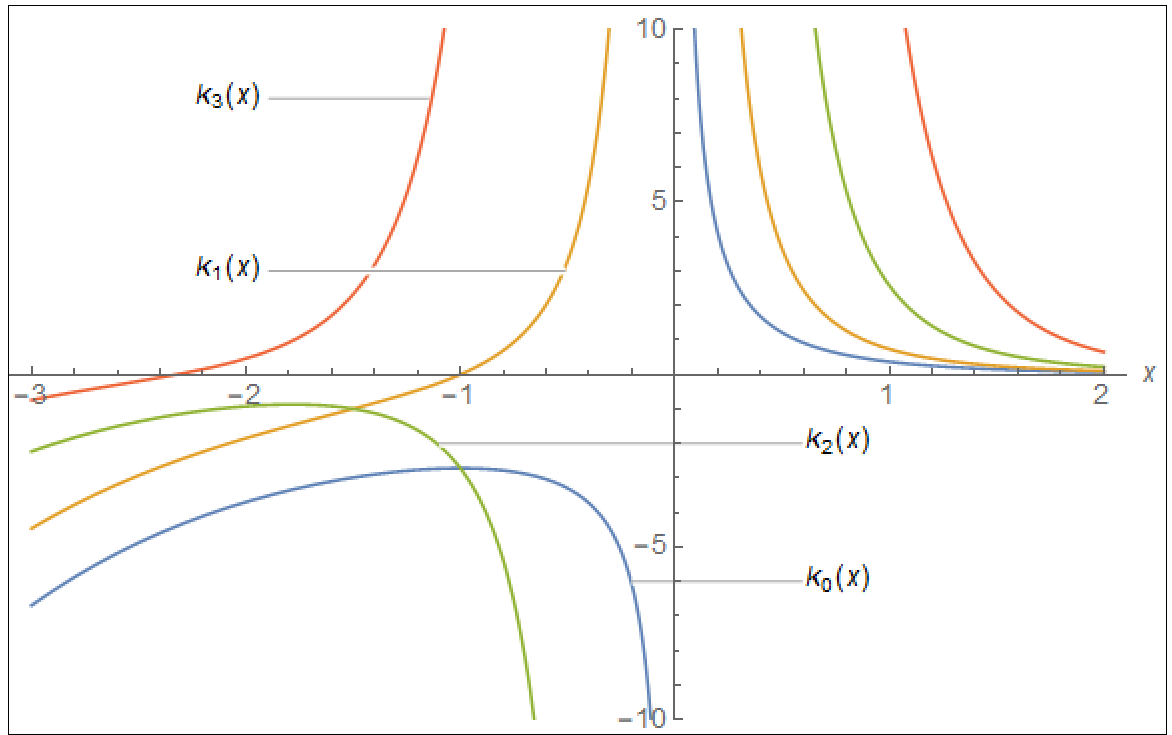}
\end{figure}

For a constant $d_{0}$ and an unknown $x$, solutions to the equation
\begin{equation}
xe^{x}=d_{0}\label{eq:LambdaWdefinition}
\end{equation}
are $x=W_{b}(d_{0})$ for all branches $b$ of Lambert $W$ that satisfy
equation (\ref{eq:LambdaWdefinition}). This function was standardized,
analyzed and named by Corless, Gonnet, Hare, Jeffrey, and Knuth \cite{CorlessEtAlLambdaW}.
Its history extends back hundreds of years and Bessel functions history
is even longer, but what has apparently not been recognized before
is that Lambert W can be expressed in terms of the inverse modified
spherical Bessel function $k_{0}(x)$ and vice versa: For nonzero
$d_{0}$ and real $x$, the first equation in Table \ref{Table:SphericalBesselToExponentialLarent}
can be converted to equation (\ref{eq:LambdaWdefinition}) by reciprocating
both sides then using the substitution $1/c_{0}\rightarrow d_{0}$.
Thus
\begin{align}
\mathrm{inverse}\,_{\boldsymbol{1}}(k_{0})(x)\:|\:x>0 & \equiv\dfrac{1}{W_{0}(x)},\label{eq:Inverse1k0}\\
\mathrm{inverse}\,_{\boldsymbol{0}}(k_{0})(x)\:|\:-\frac{1}{e}\leq x<0 & \equiv\dfrac{1}{W_{0}(x)},\label{eq:Inverse0k0}\\
\mathrm{inverse}\,_{\boldsymbol{-1}}(k_{0})(x)\:|\:-\frac{1}{e}<x<0 & \equiv\dfrac{1}{W_{-1}(x)}.\label{eq:InverseMinus1k0}
\end{align}
From Table \ref{Table:SphericalBesselToExponentialLarent} and Figure
\ref{Figure:Plotk0thruk3}, it is evident that $\mathrm{inverse}\,_{\boldsymbol{0}}(k_{0})(0)$
has infinite magnitude, as does $1/W_{0}(0)$. This together with
equations (\ref{eq:Inverse1k0}) through (\ref{eq:InverseMinus1k0})
imply that

\[
W_{\boldsymbol{0}}(x)\equiv\begin{cases}
\dfrac{1}{\mathrm{inverse}\,_{\boldsymbol{1}}(k_{0})(x)}, & x>0;\\
0 & x=0;\\
\dfrac{1}{\mathrm{inverse}\,_{\boldsymbol{0}}(k_{0})(x)}, & -\frac{1}{e}\leq x<0;
\end{cases}
\]
and
\[
W_{\boldsymbol{-1}}(x)\:|\:-\frac{1}{e}\leq x<0\equiv\dfrac{1}{\mathrm{inverse}\boldsymbol{\,_{-1}}(k_{0})(x)}.
\]

Thus for nonzero $x$, the mutual inverse $x\,e^{x}$ of $W_{0}(x)$
and $W_{-1}(x)$ is also the mutual inverse of the corresponding \textsl{three}
branches of the reciprocal of $1/k_{0}(x).$ The first-order pole
at $k_{0}(0)$ versus none for $xe^{x}$ split branch 0 of Lambert
$W$ into branches 0 and 1 of inverse spherical Bessel $k_{0}(x)$.
However, all nonzero real solutions expressible in terms of Lambert
$W_{0}(x)$ or $W_{-1}(x)$ are also expressible in terms of $\mathrm{inverse}\,_{\boldsymbol{b}}(k_{0})(x)$
with $b\in\{-1,0,1\}.$

Because of this relationship between Lambert $W$ and $\mathrm{inverse}\boldsymbol{\,_{b}}(k_{0})$,
we can regard $\mathrm{inverse}\boldsymbol{\,_{b}}(k_{n})$ for integer
$n>0$ as generalizations of Lambert $W$.

To summarize for spherical Bessel $k_{n}(x):$
\begin{prop}
For real $x$ and an equation of the form
\begin{equation}
p(x)e^{-x}=c_{0}\label{eq:knOfxEqc0}
\end{equation}
from row $n$ of the countably infinite number of rows partially listed
in Table that can be extended by recurrence $\mathrm{(}$\ref{eq:RecurrenceForkn}$\mathrm{)}$,
let $x_{n,m}$ denote the $m^{\mathrm{th}}$ infsupum abscissa of
$the,$ left side and let $B$ denote the set of all real branches
b of $\mathrm{inverse}\,_{b}(k_{n})$ for which the interval
\[
[k_{n}(x_{n,\,b-1}),k_{n}(x_{n,b}))
\]
contains $c_{0}.$

Then all of the real solutions to equation $(\ref{eq:knOfxEqc0})$,
if any, are
\[
x=\mathrm{inverse}\,_{b}(k_{n})(c_{0})
\]
where branches $b$ are all elements of set $B$.
\end{prop}

\begin{proof}
$\,$Similar to that of Proposition \ref{thm:yn}.
\end{proof}
Knowing the correspondence between Lambert $W_{0}(x)$, $W_{-1}(x)$
and $\mathrm{inverse}\,_{b}(k_{\boldsymbol{0}})(x)$ might enable
us to exploit identities for $k_{0}$ to transform more equations
to a form solvable by an explicit closed form expression. It also
might similarly facilitate integration, summation of infinite series,
and other operations with these functions.

However, a potentially greater benefit of these closed-form representations
of $k_{n}(x)$ is that for $n>0$, if we can transform any equation
to the form
\[
\mathrm{any}\:\mathrm{table}\:\mathrm{entry\:for\:}k_{n}(x)=\mathrm{any}\:\mathrm{constant}\:\mathrm{in}\:\mathrm{the}\:\mathrm{range}\:\mathrm{of}\:\mathrm{that}\:\mathrm{entry},
\]
then we can express a solution in terms of branch $b$ of $\mathrm{inverse}\,_{b}(k_{n})(x)\,$,
where $n$ is the corresponding entry in the column labeled $n$ and
$b$ is a branch that contains the float approximation of the constant.

References \cite{MaignanAndScott,MezoAndBasricz} describe how nested
instances of Lambert $W$ can be used to express explicit exact closed-form
real solutions to equations transformable to
\[
e^{-cx}=a_{0}(x-r_{1})^{m_{1}}(x-r_{1})^{m_{2}}\cdots(x-r_{n})^{m_{n}}
\]
where $c$, $a_{0}$ and $r_{1}$ through $r_{n}$ are exact real
constants, with $m_{1}$ through $m_{n}$ integer. They also describe
some physics applications. This complements the inverse spherical
Bessel $k_{n}(x)$ solutions discussed in this section, which are
not limited to equations exactly factorable into linear factors having
all real zeros, but are limited to specific integer coefficients after
factoring out an appropriate unit times the gcd of the coefficients.

As an example, Figure \ref{Figure:k2eq3On3Ln2MinusPi} plots both
sides of the equation
\[
\left(\dfrac{1}{x}+\dfrac{3}{x^{2}}+\dfrac{3}{x^{3}}\right)e^{-x}=\dfrac{3}{3\log2-\pi}\approxeq-2.82446\,.
\]

\begin{figure}[H]
\begin{centering}
\caption{Both sides of the equation $\left(\dfrac{1}{x}+\dfrac{3}{x^{2}}+\dfrac{3}{x^{3}}\right)e^{-x}=\dfrac{3}{3\log2-\pi}$}
\label{Figure:k2eq3On3Ln2MinusPi}
\par\end{centering}
\centering{}\includegraphics{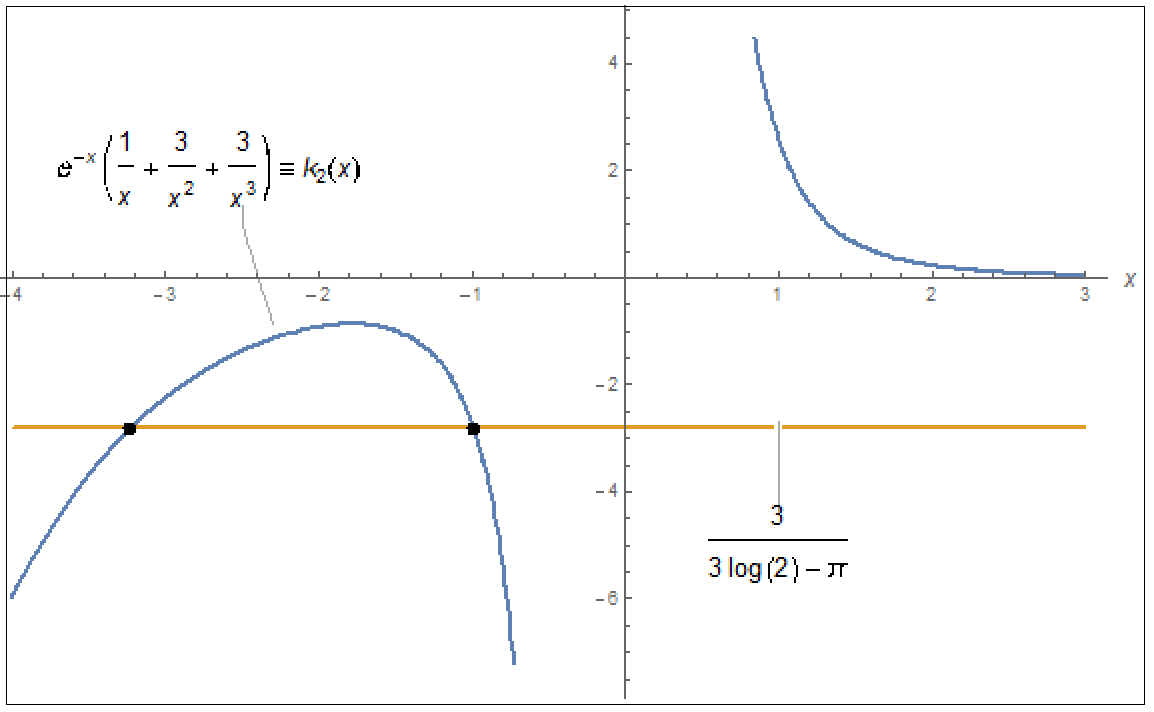}
\end{figure}
This equation has two exact closed-form real solutions: 
\[
\dfrac{1}{x=\mathrm{inverse}\,_{b}\left(k_{2}\right)\left(\dfrac{3}{3\log2-\pi}\right)}
\]
with branch $b=-1$ for $x\leq-1.78324$ where the monotonicity changes
and branch $b=0$ for $-1.78324<x\:<\:0$$\,.$ (There are no solutions
for branch $b=1$ where $x>0\,.$)

I am unaware of a published example solvable by $\mathrm{inverse}\,_{b}\left(k_{n}\right)$
for $n>0$, and I would greatly appreciate learning of some.

\section{Conclusions\label{sec:Conclusions}}
\begin{enumerate}
\item It is pleasantly surprising that four classes of equations long-thought
to have no exact closed form solutions actually do:
\begin{enumerate}
\item I found about 180 published examples solvable by $\mathrm{inverse}\,_{b}(y_{n})$.
\item I found about 150 published examples solvable by $\mathrm{inverse}\,_{b}(j_{n})$.
\item Although I have not yet found published examples of strict Laurent
$\sinh$ $\cosh$ equation solvable by $\mathrm{inverse}\,_{b}(i_{n})$,
I expect that there are some.
\item Although the many published examples solvable by $\mathrm{inverse}\,_{b}(k_{0})$
are more conveniently solvable by $W_{0}$ or $W_{-1}$ and I have
not yet found published examples solvable by $\mathrm{inverse}\,_{b}(k_{n})$
for $n>0,$ I expect that there are some.
\end{enumerate}
\item The general technique that enabled the exact solutions discussed in
this article was to recognize when the left side of an equation in
the form
\[
\mathit{expression}=\mathit{constant}
\]
is an exact transformation of some named $f(x)$ for which we can
compute and concisely represent inverses. This recognition is often
difficult for humans and more difficult to implement in computer algebra
systems than transforming $f(x)$ to a cryptic equivalent. For example,
Mathematica 12.1 has $\mathsf{EllipticNomeQ}$ and $\mathsf{InverseEllipticNomeQ}$
functions, and
\[
\mathrm{\mathsf{FunctionExpand}}\,\left[\mathrm{\mathsf{EllipticNomeQ}}\,[x]\right]\mapsto e^{-\dfrac{\pi\,\mathrm{\mathsf{EllipticK}}\,[1-x]}{\mathrm{\mathsf{EllipticK}}\,[x]}},
\]
but neither the $\mathsf{Solve}$ nor the more comprehensive $\mathsf{Reduce}$
function can determine that a solution to
\[
e^{-\dfrac{\pi\,\mathrm{\mathsf{EllipticK}}\,[1-x]}{\mathrm{\mathsf{EllipticK}}\,[x]}}=\dfrac{1}{5}
\]
is
\begin{equation}
\mathrm{\mathsf{InverseEllipticNomeQ}}\,\left[\dfrac{1}{5}\right].\label{eq:InverseEllipticNomeOf1fifth}
\end{equation}
However, executing
\[
\mathrm{\mathsf{FindRoot}}\,\left[e^{-\dfrac{\pi\,\mathrm{\mathsf{EllipticK}}\,[1-x]}{\mathrm{\mathsf{EllipticK}}\,[x]}}=\dfrac{1}{5},\left\{ x,0.5\right\} ,\mathrm{\mathsf{WorkingPrecision}}\rightarrow16\right]
\]
returns $\{x\rightarrow0.9658521935950790\}$; and entering that constant
into AskConstants returns the candidate expression (\ref{eq:InverseEllipticNomeOf1fifth})
with an excellent assessment of being the limit as the precision approaches
infinity. \textsl{Such constant-recognition software is best at recognizing
low-complexity constants such as expression} (\ref{eq:InverseEllipticNomeOf1fifth}),
\textsl{and low complexity exact closed forms are the most useful}.
Therefore routine use of such tools can lead to surprising sought
or unsought discoveries \textendash{} so much so that I would like
to find a programmer who could arrange for several such tools to run
in the background, automatically checking all of the floats in my
top-level results from all of my mathematical software as a low priority
background task.\footnote{The reason for using several is that most constant recognition programs
can propose candidates that none of the other can propose.} For example, I could imagine having an optionally open window that
logged floats and their session location paired with highly likely
exact limits thereof \textendash{} even if entries appeared noticeably
after I had moved on to further immediate calculations.
\item The utility of computer algebra systems would benefit greatly from
more multi-branched inverse special functions, even if it initially
necessitates limiting the domain to the reals and limiting orders
to appropriate integers.
\item Judging from the number of different published equations currently
known to me that were widely thought to have no explicit exact closed-form
solution but do using the techniques described in this article:
\begin{enumerate}
\item The inverses of $y_{n}(x)$ for strict Laurent cosines and sines has
the most impact, followed by those of $j_{n}(x)$ for strict Laurent
sines and cosines
\item I suspect that there are published examples solvable by inverses of
$k_{n}$ for $n>0$ and by $i_{n}$, but the fact that I have not
yet found any suggests that these solutions are significantly less
frequent.
\end{enumerate}
\end{enumerate}
The current AskConstants version 5.0 \cite{StoutemyerAskConsants}
contains implementation of the multi-branched inverse spherical Bessel
functions $y_{n}$ and $j_{n}.$ The files are ASCII text files that
can be viewed with any ASCII text editor. I plan for the next version
of AskConstants to contain implementations of the spherical Bessel
functions $i_{n}$ and $k_{n}$ together with their inverses. Perhaps
AskConstants will lead me to published examples for those functions
too. 

\subsection*{Current Limitations}

This is a ``truth in advertising'' subsection intended to reduce
the chance of misunderstanding:
\begin{enumerate}
\item Although spherical Bessel $k_{n}(x)$ and the concept of multi-branched
inverse functions precede the name Lambert $W$, I do not propose
replacing Lambert $W_{0}$ and $W_{-1}$ with $\mathrm{1/inverse}\,_{b}(k_{0})$,
for which the reciprocal introduces an inconvenient pole, an extra
branch, and a less aesthetic piecewise definition into the equivalent
of $W_{0}(x).$ The benefit of inverse spherical $k_{0}(x)$ is merely
that it provides an alternative viewpoint that might enable new identities
to be applied. The benefit of other order $\mathrm{inverse}\,_{b}(k_{n})(x)$
is that they can solve equations that appear to have no other known
explicit exact closed form solution.
\item It is important to realize that although all of the inverse spherical
Bessel functions discussed in this article can solve solve equations
containing certain linear combinations of cofactors having a certain
form, the techniques described here cannot solve equations having
\textsl{arbitrary} linear combinations of such cofactors. The techniques
described in this article exploit only one degree of freedom by transposing
all terms depending on the unknown and $\cos x$, $\sin x$, $\cosh x$,
$\sinh x$, or $e^{-x}$ to the left side and all other terms to the
right side, then dividing both sides by an appropriate unit times
the $\gcd$ of the coefficients on the left side, making the left
side coefficients be integers; then matching all of the terms on the
left but permitting one arbitrary constant $c_{0}$ on the right side.
\item Many applicable published equations contain positive rather than negative
integer powers of the variable requiring multiplication by a power
of $x;$ and many equations contain functions that must be converted
to $\cos x$, $\sin x$, $\cosh x$, $\sinh x$ and $e^{-x}$ to achieve
a form partially listed in Tables \ref{Table:SphericalBesselyn(z)ToTrigLarent-1},
\ref{Table:SphericalBessejnlToTrigLarent}, \ref{Table:SphericalBesselinToSinhCoshLaurent},
and \ref{Table:SphericalBesselToExponentialLarent}. This requires
some manual or semi-manual effort until such steps are implemented
in your computer algebra system Solve function. Moreover, during those
manual operations, it is important to beware of and account for introducing
or annihilating solutions $x=0$.
\end{enumerate}
That said, I hope that you and others will devise ways to use the
ideas here more flexibly.

\section*{Acknowledgments}

Thank you Bill Gosper for nudging me for an explanation of the proposed
closed form for the Dottie number. Thank you Rob Corless and Christophe
Vignat for your encouragement and suggestions. I also thank all of
the early pioneers of constant recognition tables and software, many
of whom are named in \cite{StoutemyerHowToHuntWildConstants}. They
foresaw the value of and developed new ways to greatly increase the
synergy between approximate and exact computation.

\end{document}